\documentclass[reqno]{amsart}
\usepackage{amssymb, stmaryrd}
\usepackage{mathrsfs}
\usepackage{amscd}
\usepackage{pdfsync} %% PDF <-> TEX

\newtheorem{theorem}{Theorem}[section]
\newtheorem{lemma}[theorem]{Lemma}

\newtheorem{example}[theorem]{Example}
\newtheorem{remark}[theorem]{Remark}
\newtheorem{proposition}[theorem]{Proposition}
\newtheorem{corollary}[theorem]{Corollary}

\newcommand{\F}{\boldsymbol{F}}
\newcommand{\f}{\boldsymbol{f}}

\usepackage{color}  %% USE OF COLORS
\definecolor{darkgreen}{rgb}{0.03, 0.5, 0.03}
 \newcommand{\bc}{\color{blue}} %

 \newcommand{\co} {\boldsymbol c}
 \newcommand{\Na} {\mbox{\texttt{Na}}}

  \newcommand{\Max} {\mbox{\rm Max}}
    \newcommand{\Spec} {\mbox{\rm Spec}}

         \newcommand{\GV} {\mbox{\texttt{GV}}}

\newcommand{\rad} {\mbox{\rm rad}}

  \newcommand{\stf} {\star{_{\!{_f}}}}
   \newcommand{\astf} {\ast{_{\!{_f}}}}

\begin{document}

\title []{$t$-local domains and valuation domains}
%[SHARPNESS \& SEMISTAR  OPERATIONS]{ SHARPNESS \& SEMISTAR  OPERATIONS\\ IN PR\"UFER-LIKE DOMAINS}
\author[]{Marco Fontana$^{(\star)}$ and Muhammad Zafrullah}
%[ ]{ M. Fontana, E.G. Houston, and  M.H. Park}}

{ \date{ \today}}
%\subjclass[2010]{13B25, 13A15, 13G05, 13B22}
%\keywords{Star and semistar operation; Nagata ring; Kronecker function ring.}
%
\thanks{$^{(\star)}$ The first named author was partially supported by {\sl GNSAGA} of {\sl Istituto Nazionale di Alta Matematica}. } 
\address{M.F.: Dipartimento di Matematica, Universit\`a degli Studi
``Roma Tre'', 00146 Rome, Italy.}
\email{fontana@mat.uniroma3.it }
\address{M.Z.: Department of Mathematics, Idaho State University, Pocatello, ID83209-8085,USA.}
\email{mzafrullah@usa.net }

\begin{abstract}
In a valuation domain $(V,M)$ every nonzero finitely generated ideal $J$ is
principal and so, in particular,  $J=J^t$, hence  the maximal ideal $M$ is  a $t$-ideal.  Therefore, the $t$-local
domains  (i.e.,  the local domains, with maximal ideal being a $t$-ideal) are ``cousins'' of valuation domains, but, as we will see in detail, not so close.                                                                                                                     Indeed, for instance, a localization of a $t$-local domain is not necessarily $t$-local,  but of course a localization of a valuation domain is a valuation domain.

So it is natural to ask under what conditions is a $t$-local domain a valuation domain?
The main purpose of the present paper is to address this question, surveying in part previous work by various authors containing useful properties for   applying them to our goal.
 \end{abstract}

\maketitle

\begin{flushright}
{\sl \Small  Dedicated to David F. Anderson}
\end{flushright}

%%%%%%%%%%%%%%%%%%%%%%%%%
%%%%%%% SECTION 1
%%%%%%%%%%%%%%%%%%%%%%%%%%
\bigskip

\section{Introduction}
%%%%%%%%%%%%%%%%%%%%%%%%%%%%%%

 We begin by reviewing the notion of a $t$-local domain.

 Let $D$ be an integral domain with quotient field $K$, let $\F(D)$ be the set
of non-zero fractional ideals of $D$, and let
$\f(D)$ be the set of all nonzero finitely generated
$D$-submodules of $K$ (obviously, $\f(D)
\subseteq \F(D)$). For $E\in \F(D),$ let $E^{-1}:=\{x\in
K \mid xE\subseteq D\}$. The functions on $\F(D)$ defined by $E\mapsto
E^{v}:=(E^{-1})^{-1}$ and $E\mapsto E^{t}:=\bigcup \{F^{v} \mid 0\neq 
F \mbox{ is a finitely generated subideal of } E \}$, called respectively the $v$-{\it operation} and  the
$t$-{\it operation} on the integral domain $D$, come under the umbrella of star operations (briefly recalled in Section 2), discussed in
Sections 32 and 34 of \cite{[G]}, where the reader can find proofs of the
basic statements made here about the $v$-, $t$- and, more generally, the star operations.

Recall that a nonzero fractional ideal $E$ of $D$ is a 
{\it $v$-ideal},  or a {\it divisorial ideal}, (resp., a {\it $t$-ideal}) if $E=E^{v}$ (resp., $E=E^{t}$)  and a {\it $v$-ideal} (resp., a {\it $t$-ideal}) {\it of finite type}
if $E=E^v=F^{v}$ (resp., $E=E^t=F^{t}$) for some finitely generated $F\in \f(D)$ and, obviously, $F \subseteq E$.
 Next, the $t$-operation
is a   star operation of finite type on the integral domain $D$,  in the sense that $E\in \F(D)$ is a $t$-ideal if and only if for each finitely generated nonzero subideal $F$ of $E$
we have $F^{v}= F^t\subseteq E$  and it {\bc is} easy to see
that if $F$ is principal $F^{v}=F=F^{t}$.

 An integral ideal of $D$ maximal with respect to being an
integral $t$-ideal is called a {\it maximal $t$-ideal} of $D$ and it is always a prime ideal. We denote by $\Max^t(D)$ the set of all the maximal $t$-ideals of $D$. This set is non empty, since
every $t$-ideal is contained in a maximal $t$-ideal,  thanks to the definition of the $t$-operation and to  Zorn's Lemma.
An integral domain is called a {\it $t$-local domain} if  it is local and its maximal ideal is a $t$-ideal.

The purpose of
this article is to survey the notion indicating what $t$-local domains are,
where they may or may not be found and what their uses are. 

The first example of a $t$-local domain that comes to mind is a valuation
domain, i.e.,  a local domain $(V,M)$ in which every nonzero finitely generated ideal is
principal. In this case, we can say that for each $F\in \f(V)$ with $F\subseteq M$ we
have $F=(a)\in M$ and so $F^{t}=(a)^{t}=(a)\subseteq M$.
But, of course, $t$-local
domains are   much more general than that. We can, for example, show that if $P$
is a height one prime ideal of an integral domain $D$, then $D_{P}$ is a $t$-local domain.
 We can show, as well will in more generality,
that if $M=pD$ is a prime ideal
generated by a prime element of a domain $D$ then $M$ is a maximal $t$-ideal
and $D_{M}$ is a $t$-local domain. 
However, we   cannot just take a prime $t$-ideal $P$ of $D$ and claim that $D_{P}$
is a $t$-local domain, as there are examples of some domains $D$ with prime 
$t$-ideals  $P$ such that $D_{P}$ is not a $t$-local domain. In Section 2, we
discuss cases of prime $t$-ideals $P$ with $D_{P}$ a $t$-local domain and
cases of domains that have prime $t$-ideals $P$ with $D_{P}$ non $t$-local,
indicating also that if $D$ is $t$-local then, for some multiplicative set $S$ of $D$,
 $D_{S}$ the ring of fractions may not be a $t$-local domain. 

Now localization may not always produce $t$-local domains, but there are
elements of a special kind whose presence in a domain $D$ ensures that $D$
is a $t$-local domain. 
 In Section 3, we record the results related
to the fact that the presence of a nonzero   nonunit comparable element (definition recalled later) in an integral
domain $D$ makes $D$ into a $t$-local domain. The related results include
for instance (1) the effects the presence of a nonzero  nonunit comparable
element on different kinds of domains, (2) the presence of a nonzero
comparable element in some domains would make them into valuation domains, 
if $D$ is Noetherian then the presence of a nonzero   nonunit comparable
element in $D$ makes $D$ a DVR (= discrete valuation ring), (3) a $t$-local domain may not have a
comparable   element,  and so on, the list continues.

Citing Krull, P.M. Cohn \cite{[C]} showed that $D$ is a valuation domain if and only if 
$
D$ is a B\'ezout domain and a local domain. (In fact, in this result
\textquotedblleft B\'ezout\textquotedblright can be replaced by \textquotedblleft Pr\"ufer\textquotedblright; 
here $D$ is B\'ezout --respectively, Pr\"ufer-- if every nonzero finitely generated ideal of $D$ is principal --respectively,
invertible--.)
 In Section 4, we show that $D$ is a valuation domain if and
only if $D$ is a GCD domain and a $t$-local domain, and point out that if, in
the above statement, we replace \textquotedblleft GCD
domain\textquotedblright\ by \textquotedblleft P$v$MD\textquotedblright\ the
result would still be a characterization of   a valuation domain (here, $D$ is a
P$v$MD, if for each pair $0\neq a,b \in D$ we have 
$((a,b)\frac{(a)\cap (b)}{ab})^{t}=D$).  But of course we do not stop here, we point to situations where
recognizing the  fact that the domain in question is a $t$-local domain
makes proving that it is a valuation domain easier.

Section 5   has to do with \textquotedblleft applications\textquotedblright\
which are essentially more efficient proofs of known results. We follow the
study of the ring called Shannon's quadratic extension in \cite{[HLOST]} and point out
that it is indeed a $t$-local domain, thus providing a shorter, more
efficient proof of Theorem 6.2 of \cite{[HLOST]}. We also point to examples of maximal $t
$-ideals $Q$ in a particular domain $D$ such that $D_{Q}$ is not $t$-local.

\bigskip
%%%%%%%%%%%%%%%%%%%%%%%%%
%%%%%%% SECTION 2
%%%%%%%%%%%%%%%%%%%%%%%%%%
\section{Background results and $t$-local domains}
%%%%%%%%%%%%%%%%%
%%%%%%%%%%%%%%%%%

We start with proving some important preliminary results. But, for that, we need to
recall the formal definition of star operation.
A  \emph{star operation} on $D$ is a map $\ast:
\boldsymbol{{F}}(D) \to \boldsymbol{{F}}(D),\ E \mapsto E^\ast$,  such that, for all $x \in K$, $x \neq 0$, and
for all $E,F \in  \boldsymbol{F}(D)$, the following
properties hold:
\begin{enumerate}
\item[$(\ast_1)$] $(xD)^\ast=xD$;
 \item[$(\ast_2)$] $E
\subseteq F$ implies $E^\ast \subseteq F^\ast$;
\item[$(\ast_3)$] $E \subseteq E^\ast$ and $E^{\ast \ast} :=
\left(E^\ast \right)^\ast=E^\ast$;
\end{enumerate}
\cite[Section 32]{[G]}).

 If $\ast$ is a star operation on $D$, then we can
consider a map\ $\astf: \F(D) \to
\F(D)$ defined, for each $E \in
\F(D)$, as follows:

\centerline{$E^{\ast_{\!_f}}:=\bigcup \{F^\ast\mid \ F \in
\f(D) \mbox{ and } F \subseteq E\}$.}

\noindent It is easy to see that $\astf$ is a star
operation on $D$, called \emph{the  finite
type star operation associated to $\ast$} (or \emph{the star operation of finite
type associated to $\ast$}).  A star
operation $\ast$ is called a \emph{finite
type star operation} (or,  \emph{star operation of finite
type})  if $\ast=\astf$.  It is easy to see that
$(\ast_{\!_f}\!)_{\!_f}=\astf$ (that is, $\astf$ is
of finite type).

If $\ast_1$ and $\ast_2$ are two star operations on $D$, we
say that $\ast_1 \leq \ast_2$ if $E^{\ast_1} \subseteq
E^{\ast_2}$, for each $E \in \F(D)$, equivalently, if
$\left(E^{\ast_{1}}\right)^{\ast_{2}} = E^{\ast_2}=
\left(E^{\ast_{2}}\right)^{\ast_{1}}$, for each $E \in
\F(D)$.
Obviously, for each star operation  $\ast$,
we
have $\astf \leq \ast$.  Clearly, $v_{\!_f} =t$. Let $d_D$ (or, simply, $d$)  be the {\it identity star operation on $D$.}
  Clearly, $d \leq \ast$  and, moreover, $\ast \leq v$, for all star  operations $\ast$ on $D$  \cite[Theorem 34.1(4)]{[G]}.

   Recall that an integral domain $D$ is called a
{\it   Pr\"ufer $v$-multiplication domain}, (for short, {\it P$v$MD}), if every nonzero finitely generated $F\in
\f(D)$ is $t$-invertible, i.e., $(FF^{-1})^t =D$.  Obviously, every Pr\"ufer domain is a P$v$MD.
It is well known (see, Griffin \cite[Theorem 5]{[Gr]}) that $D$ is a
P$v$MD if and only if $D_{Q}$ is a valuation domain, for each maximal (or, equivalently, prime) $t$-ideal  $Q$ of $D$.

Any unexplained 
terminology is straightforward,   well accepted, and usually comes from {\cite{[Kap]}
or \cite{[G]}.

\bigskip
%%%%%%%%%  LEMMA \label{HH}
\begin{lemma}\label{HH} \emph{(Hedstrom-Houston \cite[Proposition 1.1]{[HH]})}
Let $\ast$ be a star operation on an  integral domain $D$ and let $\astf$ be the finite type star operation on $D$ canonically associated with $\ast$. If $P$ is a minimal prime ideal over a $\astf$-ideal of $D$, then $P$ is a $\astf$-ideal.
\end{lemma}
%%%%%%%
\begin{proof} 
Let $J$ be  a finitely generated (integral) ideal contained in $P$, the conclusion will follow if we show that $J^\ast \subseteq P$.
Since $P$ is minimal over some (integral) ideal $I$, with $I = I^{\astf}$, then $\rad(ID_P) = PD_P$ and, since $J$ is finitely generated,   there exists an integer  $m\geq 1$ such that $J^m D_P \subseteq ID_P$. Therefore, for some $s \in D\setminus P$,
$ s J^m \subseteq I$. Thus, $s(J^\ast)^m \subseteq s(J^m)^\ast  = s(J^m)^{\astf} \subseteq I^{\astf} = I\subseteq P$, and so  $J^\ast \subseteq P$, since $s \notin P$.
\end{proof}
%%%%%%%%

The next step is to apply  this lemma  for obtaining some sufficient conditions for a local domain to be a   $t$-local domain (recall that an integral domain is  a  $t$-local domain if  it is local and its maximal ideal is a $t$-ideal).

%%%%%% REMARK
\begin{remark} 
\em{(1) Note that if $D$ is an integral domain such that $\Max^t(D)$ contains only one element, then  $D$ is necessarily a $t$-local domain (and conversely).
If not, let  $M$   be the unique $t$-maximal ideal of $D$ and $N$ be a maximal ideal of $D$ with $N \neq M$. Let $x \in N\setminus M$, clearly, the $t$-ideal $xD$ must be contained in some $t$-maximal ideal.  
In the present situation $xD$ should be contained in $M$ and this is a contradiction.

(2) Note that if $D$ is   a local domain with divisorial maximal ideal, then clearly $D$ is $t$-local. The converse is not true: take, for instance, a valuation domain with nonprincipal maximal ideal (e.g., a 1-dimensional  non-discrete valuation domain).

 (3) In an integral domain $D$, the set of maximal divisorial ideals, $\Max^v(D)$, might be empty (e.g., take a 1-dimensional valuation domain with nonprincipal maximal ideal). However, if 
$\Max^v(D) \neq \emptyset$, a maximal divisorial ideal is a prime $t$-ideal, but it might be a nonmaximal $t$-ideal (for explicit examples see \cite{[GR]}, where the problem of when a maximal divisorial ideal is a 
maximal $t$-ideal is investigated).
}
\end{remark}
%%%%%%%%%%

%%%%%%%% COROLLARY  \label{A-AA-AB-AC}
\begin{corollary} \label{A-AA-AB-AC} Let $D$ be a local domain with maximal ideal $M$.
Then, $D$ is $t$-local  in each of the following situations.
\begin{itemize}
\item[(1)] \label{A} The maximal ideal $M$ is minimal over (i.e., is the radical of) an integral $t$-ideal of $D$.

\item[(2)] \label{A'} The maximal ideal $M$ is an associated prime over a principal ideal of $D$ (i.e., there exist $a \in D$ and $b \in D\setminus aD$ such that $M$ is minimal over $(aD:_D bD)$).  

\item[(3)]\label {AA}  The maximal ideal $M$ is
minimal over (i.e., is the radical of ) a principal ideal of $D$.

\item[(4)]\label{AB} The maximal ideal $M$ is
principal.

\item[(5)]\label{AC}   The integral domain $D$ is $1$-dimensional.
\end{itemize}
\end{corollary}
%%%%%%%%%%%%%%
\begin{proof} (1) is a straightforward consequence of Lemma \ref{HH}.
 (2) and (3) are obvious from (1), because a proper ideal of the type $(aD:_D bD)$ and a principal ideal  are both $t$-ideals.
(4) is  trivial consequence of (3).
Finally, (5) follows from the fact that, in this case, the maximal ideal is a minimal prime
over every nonzero (principal)  ideal contained in it.
\end{proof}
%%%%%%%%%%

%%%%%  PROPOSITION \label{AD}
\begin{proposition} \label{AD}
If $(D, M)$ is  a  local domain and  the prime ideals of $D$ are  comparable in pairs,
 i.e., $\Spec(D)$ is
linearly ordered under inclusion, then $D$ is $t$-local.
\end{proposition}
%%%%%%%%%%%%%
\begin{proof}
Let $I=(x_{1},x_{2}, \dots, x_{n})$  be a nonzero proper finitely generated ideal of $D$ and let 
$P$ be a minimal prime of $I$.
The prime spectrum $\Spec(D)$ being linearly ordered forces $P$ to be
unique. Now let, for each $i=1, 2,\dots, n$, $P(x_{i})$ be the minimal prime of the principal ideal
$(x_{i})$. Again, by the linearity of order of $\Spec(D)$, for some $1\leq k\leq n$, $P(x_{k})\subseteq P(x_{j})$ for  all $j\neq k$.
 So $P(x_{k})\supseteq I$ and
so $P(x_{k})\supseteq P$.  But as $x_{k}\in P$, $P(x_{k})\subseteq P$.
 Whence every proper nonzero finitely generated ideal of $D$ is contained in a prime
ideal of $D$ that is minimal over a principal ideal and, hence, $P$  is a $t$-ideal, by Corollary \ref{A-AA-AB-AC}(1).
Thus, $I^{v}=I^t \subseteq
P\subseteq M$. Since $I$ is arbitrary as a finitely generated proper ideal of $D$, $M$ is
a $t$-ideal.
\end{proof}
%%%%%%%%%%%%%

%%%%% REMARK
\begin{remark}
\em{
Note that, {\sl mutatis mutandis}, from the proof of the previous proposition, if $\Spec(D)$ is
linearly ordered under inclusion, we do not deduce only that $D$ is $t$-local,  but also that every prime ideal of $D$ is a $t$-ideal (see also   \cite[Theorem 3.19]{[K-89]}).
}
\end{remark}
%%%%%%%

 It is known that if $J$ is a $t$-ideal of a ring of fractions $D_S$ of an integral domain $D$ with respect to a multiplicative subset $S$ of $D$, then $J\cap D $ is a $t$-ideal of $D$   \cite[Lemma 3.17(1)]{[K-89]}.  However, $I$ being a $t$-ideal of the integral domain $D$ does not imply, in general,  that $ID_S$ is a $t$-ideal of $D_S$, even though $ID_S \cap D$ is a $t$-ideal of $D$  \cite[Lemma 3.17(2)]{[K-89]}
In particular, as the following Example \ref{bad behaved} will show, the prime $t$-ideals may have a ``bad behaviour'', that is  if $P$ is a prime $t$-ideal of $D$ then $PD_S$ may not be a prime $t$-ideal for some multiplicative set $S$ disjoint with $P$.

 The authors of \cite{[MZ]} were led to this conclusion seeing an example given by W. Heinzer and J. Ohm   \cite{[HO]} of an essential domain (i.e., an integral domain  $D = \bigcap D_P$  where $P$ ranges over prime ideals of $D$ such that $D_P$ is a valuation domain) that is not a P$v$MD.
The reason for this conclusion came from the following observation. For each maximal ideal $M$ of the Heinzer-Ohm example $D$, $D_M$ is a unique factorization domain, meaning the Heinzer-Ohm example is a locally GCD domain. 
 Now, if for each maximal $t$-ideal $Q$, $QD_Q$ were a prime $t$-ideal of $D_Q$, then $D_Q$ would be a $t$-local domain  and a GCD domain. 
 But, as we shall see in the following  Proposition \ref{F}, a $t$-local GCD domain is a valuation domain. So, we would have $D_Q$ a valuation domain, for every maximal $t$-ideal $Q$ of $D$, making $D$ a P$v$MD.  Therefore, since in this example $D$ is not a P$v$MD,  $Q D_Q$ might not be a $t$-ideal, for some maximal $t$-ideal $Q$ of $D$.
  Indeed, an integral domain $D$ which is locally a P$v$MD is a P$v$MD if and only if $QD_Q$ is a $t$-ideal for every maximal $t$-ideal $Q$ of $D$.

In \cite{[Z-WB]}, a {\it prime ($t$-ideal) $P$} in an integral domain $D$ was called {\it well behaved} if $PD_P$ is a prime $t$-ideal of $D_P$. 
We say that  an {\it integral domain} $D$ is {\it well behaved} if every prime ($t$-ideal) of $D$ is well behaved. 
In \cite{[Z-WB]},  M. Zafrullah characterized well behaved domains and showed that most of the known domains, including P$v$MDs, are well behaved.
Furthermore, in the same paper, there is also an example of an integral domain $D$ such that every $Q\in \Max^t(D)$ is well behaved, but $D$ is not well behaved.
This example is obtained by a pullback construction, as briefly recalled below (for the details of the proofs see \cite{[Z-WB]}).

%%%%%%%% EXAMPLE \label{bad behaved} 
\begin{example}\label{bad behaved} 
\em{
Let $(V, M)$ be a valuation domain with $\dim(V) \geq 2$ and let $P$ be a nonzero nonmaximal prime ideal of $V$, set $D := V + XV_P[X]$.
In  \cite[Lemma 2.3,  2.4, and Proposition 2.5]{[Z-WB]}, it is proved that
$$
\Max^t(D) = \{ fD \mid f \in D, 
\ f  \mbox{ is a prime element of } D \mbox{ such that } f(0) \in V\setminus M \} \cup \{ N\},
 $$
where  $ N :=\{ f\in D\mid f(0) \in M\}= M + XV_P[X]$ is a maximal ideal of $D$.

By the previous description of $
\Max^t(D)$, it is not hard to see that, for each $Q \in \Max^t(D)$, $QD_Q$ is a maximal $t$-ideal of $D_Q$.
Now, we consider the prime ideal $\frak{P} := P + XV_P[X]$ of $D$. Since $P = \bigcap \{aV \mid a \in M\setminus P\}$, a direct verification shows that $ \frak{P} = \bigcap \{aD \mid a \in M\setminus P\}$.
Thus  $ \frak{P} $ is a $v$-ideal and, in particular, a $t$-ideal of $D$. However,  after observing that $\frak{P} \cap (V\setminus P) = \emptyset$, and so $ D_{\frak{P}} = (V + XV_P[X])_{P+ XV_P[X]} = (V_P[X])_{\frak{P}V_P[X]}$  and $\frak{P}V_P[X]= PV_P + XV_P[X]$,   it can be shown that  $\frak{P}D_{\frak{P}} = 
\frak{P}{V_{P}[X]}_{\frak{P}V_{P}[X]}$
  is not  a $t$-ideal of $D_{\frak{P}}$.
}
\end{example}
%%%%%%%%%%%%%%
\medskip
   
By the previous observations and example, for each $P \in \Spec(D)$, if $D_P$ is a $t$-local domain, then $P$ is a $t$-prime ideal of $D$; on the other hand,  if  a prime ideal $P$ is a $t$-ideal  of $D$,  it is not true, in general, that $D_P$ is a $t$-local domain.
We give next some sufficient conditions for the localizations of an integral domain to be $t$-local domains.

%%%%%  PROPOSITION \label{local t}
\begin{proposition} \label{local t}
Let $D$ be an integral domain.
\begin{itemize}
\item[(1)]  If $Q$ is an associated prime ideal over a principal ideal of $D$, then $D_Q$ is a $t$-local domain.
\item[(2)]  If $Q \in \Max^t(D)$ and $Q$ is a potent ideal (i.e., it contains a nonzero finitely generated ideal that is not contained in any other maximal $t$-ideal), then $D_Q$ is a $t$-local domain.
\item[(3)]  If $D$ has the finite $t$-character (i.e., every nonzero nonunit element of $D$ belongs to at most a finite number of  maximal $t$-ideals), then $D_Q$ is a $t$-local domain, for each $Q \in \Max^t(D)$.

\end{itemize}
\end{proposition}
%%%%%%%%
\begin{proof}
(1) Since $Q$ is minimal over a $t$-ideal of $D$ of the type $(aD:_D bD)$, $QD_Q$ is minimal over the ideal   $(aD:_D bD)D_Q = (aD_Q:_{D_Q} bD_Q)$, which is a $t$-ideal of $D_Q$, and thus  $QD_Q$is a $t$-ideal  of  $ D_Q$ (Corollary \ref{A-AA-AB-AC}(2)). 

(2) was proven in \cite[Theorem 1.1(1)]{[ACZ]} and (3) follows from (2), since each maximal $t$-ideal in an integral domain with finite $t$-character is potent  \cite[Theorem 1.1(2)]{[ACZ]}.
\end{proof}
%%%%%%%

%%%%%%%%%%  REMARK
\begin{remark}
\em{ Recall that a prime $t$-ideal $P$ of an integral domain $D$ is said to be a {\it $t$-sharp ideal} if  $\bigcap \{D_Q \mid Q\in \Max^t(D),  P \nsubseteq Q\} \not\subseteq D_P$ \cite[Section 3]{[HFP]}.
For a P$v$MD, it is known that  a  prime $t$-ideal $P$ is $t$-sharp  if and only if it is potent  \cite[Proposition 3.1]{[HFP]}.
}
\end{remark}
%%%%%%%%%%%%%%%%%%

If  $D$ has the finite $t$-character, then every  maximal $t$-ideal is well behaved (Proposition \ref{local t}(3)). 
It was observed in \cite[Example 3.9]{[ACZ]}  that  the integral   domain $D$, described in Example \ref{bad behaved},  has the finite $t$-character and so even an integral domain with the finite $t$-character might not be well behaved. We provide next another example of an integral domain which   happens  to be $t$-local (and so, tri\-vially, with the finite $t$-character) and it is not well behaved (see, also, \cite[Remark 3.2(2)]{[ACZ]}).

%%%%%%%% EXAMPLE \label{bad behaved 2} 
\begin{example}\label{bad behaved 2} 
\em{
Let $D_1:= \mathbb{Z}_{(p)}$ and so $D_1$ is a rank 1 discrete valuation domain of the field of rational numbers $K_1:=\mathbb{Q}$, with maximal principal ideal $N_1:= p\mathbb{Z}_{(p)}$.

Let $D_2:=\mathbb{Q}[\![X,Y]\!]$ be the power series ring in two variables with coefficients in the field  $\mathbb Q$. Clearly, $D_2$ is  an integrally closed local Noetherian 2-dimensional integral domain with maximal ideal $N_2 := (X,Y)\mathbb{Q}[\![X,Y]\!] $ and field of quotients $K_2 := \mathbb{Q}(\!(X,Y)\!)$.
Let $D_3 = K_2[\![Z]\!] = \mathbb{Q}(\!(X,Y)\!)[\![Z]\!]$;  $D_3$ is a rank 1 discrete valuation domain of the field $K_3 := K_2(\!(Z)\!)$, with maximal ideal $N_3 := ZK_2[\![Z]\!]$.  
 Set  
 $$
 D := D_1+ N_2 +N_3 =  \mathbb{Z}_{(p)} + (X,Y)\mathbb{Q}[\![X,Y]\!] + Z \mathbb{Q}(\!(X,Y)\!)[\![Z]\!]\,.$$
 Clearly, $D \subset T:= D_2+N_3= \mathbb{Q}[\![X,Y]\!] + Z \mathbb{Q}(\!(X,Y)\!)[\![Z]\!] \subset  D_3 =K_2+N_3 = \mathbb{Q}(\!(X,Y)\!)[\![Z]\!] $.
By well known properties of rings arising from pullback constructions, it is not hard to see that the following hold.
\begin{itemize}
\item[(1)] $T$ is a 3-dimensional local ring with maximal ideal $Q:=N_2+N_3$ and the locali\-zations of $T$ at each one of its
infinitely many prime ideals of height 2 is a rank 2 discrete valuation domain.
\item[(2)] $T$  has  unique prime ideal of height 1, that is $N_3$. More precisely, $N_3$ is a  common prime ideal of $T$ and $D_3$ and  $N_3 = (T:D_3)$, since $N_3$ is the maximal ideal of the local domain $D_3$; therefore, $N_3$ is a $t$-ideal (in fact, a $v$-ideal) of $T$. Furthermore, $T_{N_3}= D_3$  is a rank 1 discrete valuation domain. 
\item[(3)] $D$ is a 4-dimensional local domain, with maximal ideal $M:= N_1+N_2+N_3$.
\item[(4)] $M$ is a $t$-ideal (in fact, a $v$-ideal) of $D$, since $M=pD$, and so $D$ is a $t$-local domain.
\item[(5)] $Q= N_2+N_3 = \bigcap\{p^nD \mid n\geq 0\}$ is the unique prime of height 3 in $D$ and it is a $t$-ideal (in fact, a $v$-ideal) of $D$, since $Q$ is a common ideal of $D$ and $T$ and, since it is the maximal ideal of $T$, $Q =(D:T)$.
\item[(6)]  For each one of the infinitely many height 2 prime ideals $P$ of $D$, there exist a unique prime ideal $P' $ of $T$ such that $P'\cap D = P$ and the canonical embedding homomorphism $D_{P} \subseteq T_{P'}$ is an isomorphism; thus $D_{P}$ is a rank 2 discrete valuation domain.
\item[(7)] Set $S:= \{ p^n\mid n\geq 0\}$, clearly $S$ is a multiplicative set of $D$ and $D_S = \mathbb Q + N_2+N_3 = \mathbb{Q} + (X,Y)\mathbb{Q}[\![X,Y]\!] + Z \mathbb{Q}(\!(X,Y)\!)[\![Z]\!]= D_Q=T\,.$ 
\item[(8)]  $QD_S =QD_Q=QT=Q$ is not a $t$-ideal of $D_Q=T$, since the elements $X, Y \in QD_Q=Q$ are $v$-coprime (note that, if $F$ is a nonzero finitely generated ideal in a $t$-ideal $I$,
then $F^{v}\subseteq I$). 
\item[(9)] By the previous properties, it follows that $T$ is a local, but not $t$-local,  P$v$MD, since the localization at all its nonzero nonmaximal prime ideals is a valuation domain and its maximal ideal $Q$  is not a $t$-ideal of $T$. Moreover, $T$ is not completely integrally closed and so it is not a Krull domain, since its complete integral closure is $D_3$, because $N_3 = (T:D_3)$. $T$ does not have the finite $t$-character, since each nonzero element inside its unique height 1 prime ($t$-)ideal $N_3$ is contained in all the infinitely many maximal  $t$-ideals, which are all its prime ideals of height 2. 
 \item[(10)]  Every nonzero prime ideal of $D$ is a $t$-ideal and all of them are well behaved, except $Q$, its unique prime of height 3 (which is a $t$-ideal of $D$, but it is not a $t$-ideal in $D_Q=T$). 
\end{itemize} 
}
\end{example}
%%%%%%%%%%%%%%%%%%%%%%
\medskip

The following result was proved by D.D. Anderson, G. W. Chang, and M. Zafrullah in 2013 \cite[Proposition 1.12(1)]{[ACZ]}:

%% PROPOSITION \label{ACZ}
\begin{proposition} \label{ACZ}    Let $D$ be a $t$-local domain, then  the following hold.
\begin{enumerate}
\item[(1)] 
Every $t$-invertible ideal  (i.e., an ideal  $I$ such that $(II^{-1})^t =D$) is principal.
 \item[(2)]  If $I$ is an ideal of $D$ such that $(I^n)^t =D$ for some $n\geq 2$, then $I$ is principal.§ 
\end{enumerate}
\end{proposition}
%%%%%%%%%
\begin{proof}
 (1) If $I$ be a $t$-invertible ideal of $D$ then $II^{-1}$ is in no maximal $t$-ideals of $D$ and this implies that $II^{-1}D_Q=D_Q$ for every $Q \in \Max^t(D)$. 
  In this special situation, $\Max^t(D)=\Max(D)= \{M\}$, where $M$  is the only maximal ideal of the $t$-local domain $D$.  Thus, $I$ is invertible in a local domain and hence it is principal.

(2) In this situation, $I$ is $t$-invertible, hence the conclusion follows from (1).
\end{proof}
%%%%%%

Note that the set \texttt{TI}$(D)$ of all the fractional $t$-invertible $t$-ideals of an integral domain $D$ is a group with respect to the operation $I {\hskip -0.2pt} {\boldsymbol \cdot}_t {\hskip -0.2pt} J := (IJ)^t$,  having as subgroup the set \texttt{Princ}$(D)$ of all nonzero fractional principal ideals of $D$. The quotient group \texttt{Cl}$^t(D) :=$ \texttt{TI}$(D)/$\texttt{Princ}$(D)$ is called the {\it $t$-class group of $D$}. The previous Proposition \ref{ACZ}  can be also stated by saying that: {\it  if $D$ is a $t$-local domain then \em{\texttt{Cl}}$^t(D)=0$.}
}

\medskip

%%%%%%%%%%%%%%%%%%%%%%%%%
%%%%%%% SECTION 3
%%%%%%%%%%%%%%%%%%%%%%%%%%
%%%%%%%%%%%%%%%%%%%%%%%%%%%%%%%%%%
\section{$t$-local domains and local dw-domains}
%%%%%%%%%%%%%%%%%%%%%%%%

 A nonzero ideal $J$ of an integral domain $D$ is called a {\it Glaz-Vasconcelos ideal} (for short, a {\it \GV-ideal}) if $J$ is finitely generated and $J^{-1} = D$. 
The set of Glaz-Vasconcelos ideals of $D$ is denoted by $\GV(D)$ \cite{[GV]}.
Given a nonzero fractional ideal  $E$ of $D$, the {\it $w$-closure of} $E$ is the fractional ideal 
$E^w := \{x\in K \mid xJ \subseteq E, \mbox{ for some } J\in \GV(D)\}$. 
A nonzero fractional ideal $E$ is called a {\it $w$-ideal}  if $E=E^w$.  The $w$-operation was introduced  by Wang-McCasland in \cite{[WMcC]}.

It is well known that $w$, like $v$, $t$, and the identity operation $d$ are examples of star operations  (respectively, $w$, like  $t$, and $d$ are examples of star operations of finite type) \cite[Proposition 3.2]{[HH]} 
and also that $d\leq w\leq t\
\leq v$, this means that, for each $E\in \F(D)$, we have the following inclusions $ E^d:= E \subseteq E^w\subseteq E^t \subseteq E^v$.  
Furthermore, for each $E\in \F(D)$, $E^w = \bigcap \{ED_Q \mid Q \in \Max^t(D)\}$ and the set of maximal $w$-ideals of $D$, $\Max^w(D)$, coincide with the set of maximal $t$-ideals of $D$, $\Max^t(D)$  \cite{[W]}.

\bigskip

It is natural to ask what is the relation between a $t$-local domain and  a {\it $w$-local domain}, i.e., a local domain such that its maximal ideal is a $w$-ideal. 
A $t$-local domain is  necessarily a $w$-local domain, since $d\leq w\leq t$  and conversely, since as observed above,  $\Max^w(D)=\Max^t(D)$.  We will show that something more is true, that is, in a $t$-local domain,  every nonzero ideal is a $w$-ideal. For showing this, we need some preliminaries.

Recall that a {\it $DW$-domain} is an integral domain $D$  such that $d=w$, i.e., for each nonzero fractional ideal  $E$ of $D$, $E =E^w$; this is  equivalent to requiring that every nonzero (integral) finitely generated ideal of $D$ is a $w$-ideal.  The following result is due to  F. Wang \cite[Proposition 1.3]{[W-CA]} (see also A. Mimouni \cite[Proposition 2.2]{Mi}).

%%%%%%%%%%%% PROPOSITION
\begin{proposition} Let $D$ be an integral domain.
The following are equivalent.
\begin{itemize}
\item[(i)] $D$ is a $DW$-domain.
\item[(ii)]  Every nonzero prime ideal  of $D$ is a $w$-ideal
\item[(iii)]  Every maximal ideal of $D$ is a $w$-ideal
\item[(iv)]  Every maximal ideal  of $D$ is a $t$-ideal
\item[(v)]  $\emph{\GV}(D) =\{D\}$.
\end{itemize}

\end{proposition}
%%%%%%%%%%
\begin{proof}
 Obviously, (i)$\Rightarrow$(ii)$\Rightarrow$(iii).

(iii)$\Rightarrow$(iv) is a consequence of  the fact that $\Max^w(D) = \Max^t(D)$.

(iv)$\Rightarrow$(v) Let $J \in \GV(D)$ and $J \subsetneq D$. Let $M \in \Max^t(D)$ such that $J\subseteq M$,  then $D =J^v=J^t\subseteq M^t=M$, which is a contradiction.

(v)$\Rightarrow$(i) Let $I$ be a nonzero ideal of $D$ and let $0 \neq x \in I^w$ then, for some $J\in \GV(D)$, $xJ \subseteq I$. Since ${\GV}(D) =\{D\}$, $xD \subseteq I$ and so $I^w \subseteq I$.
\end{proof} 
%%%%%%%%%%%%

From the previous proposition we deduce immediately the following.

%%%%%%%%%%%% COROLLARY \label{w=t} 
\begin{corollary}\label{w=t} Let $D$ be an integral domain.
The following are equivalent.
\begin{itemize}
\item[(i)]    $D$ is a $t$-local.
\item[(ii)]    $D$ is a $w$-local
\item[(iii)]  $D$ is a local $DW$-domain.
\end{itemize}

\end{corollary}
%%%%%%%%%%%

%%%% REMARK \label{heinzer} 
\begin{remark}\label{heinzer}  \emph{
Note that, for a $t$-local domain, it is not true that every nonzero ideal is a $t$-ideal, i.e.,  a domain such that $d=t$  or a {\it $DT$-domain};  even more, for a $t$-local domain, it may happen  that every nonzero prime ideal is a $t$-ideal,  without being a $DT$-domain (see the following Example \ref{DW-DT}).
The $DT$-domains are also called {\it fg$v$-domains}, that is domains such that every nonzero finitely generated ideal is a $v$-ideal since, for each nonzero ideal $I$, $I =I^t$ if and only if, for each nonzero finitely generated ideal $J $, $J^v=J^t =J$.
M. Zafrullah in \cite{[Z-fgv]} studied the fgv-domains  and he proved that an integrally closed fgv-domain is a Pr\"ufer domain.
Note that, for a Noetherian domain, being a $DT$-domain is equivalent to being a domain such that each nonzero ideal is divisorial (i.e.,  a domain such that $d=v$). In particular,  W. Heinzer  has proven that,  for a Noetherian domain $D$, if every nonzero ideal is divisorial, then $\dim(D) \leq 1$ \cite[Corollary 4.3]{[H]};  furthermore, for an integrally closed Noetherian domain (or, more generally, for any completely integraly cosed domain) $D$, every nonzero ideal is divisorial if and only if  $D$ is Dedekind domain \cite[Proposition 5.5]{[H]}.
\newline
Finally, note   that $DT$-domains are exactly the $DW$-domains that are at the same time {\it $TW$-domains}, i.e., domains such that $w=t$ \cite{Mi-2003}.
}
\end{remark}

%%%% LEMMA
\begin{lemma} Let $(T,N)$ be a local domain, let $\boldsymbol{k}(T) := T/N$, let $\varphi: T \rightarrow \boldsymbol{k}(T)$ be the canonical projection, and let $R$ be  a subring of the field $\boldsymbol{k}(T)$.
Set $D :=\varphi^{-1}(R)$.  then  $D$ is a $t$-local domain with maximal ideal $M$ if and only if $R$ is a $t$-local domain (with maximal ideal $\varphi(M)$).  
\end{lemma}
%%%%%%%%%%%%%
\begin{proof} By the standard properties of the pullbacks constructions,  $D$ is a local domain with maximal ideal $M$ if and only if $R$ is  a local domain (with maximal ideal $\varphi(M)$)  \cite[Corollary 1.5]{[F]}.
Moreover, for each $E \in \F(R)$, $\varphi^{-1}(E) \in \F(D)$ and  $(\varphi^{-1}(E))^w =  \varphi^{-1}(E^w)$   \cite[Lemma 3.1]{Mi-2003}.
Note that $M=\varphi^{-1}( \varphi(M))$, and thus $M=M^w$ if and only if $\varphi(M) =(\varphi(M))^w$. Therefore $(D,M)$ is $w$-local if and only if $(R, \varphi(M))$ is $w$-local. The conclusion follows from Corollary \ref{w=t}. 
\end{proof}
%%%%%%%%

%%%%%% EXAMPLE \label{DW-DT}
\begin{example}\label{DW-DT} {\sl Example of a Noetherian  $t$-local domain (hence, a local $DW$-domain) which is not a $DT$-domain, but each nonzero prime ideal is a $t$-ideal. }

\rm{
Consider the 2-dimensional Noetherian integrally closed domain $T:=\mathbb{C}[X,Y]_{(X,Y)}$, which is clearly not a $t$-local domain, since its (finitely generated maximal ) ideal  $M:=(X,Y) \mathbb{C}[X,Y]_{(X,Y)}$ is not a divisorial ideal of $T$ (the only divisorial ideals of $T$ are its height 1 prime ideals).
However, by the previous lemma, the local 2-dimensional Noetherian domain $D:= \mathbb{R} +(X,Y) \mathbb{C}[X,Y]_{(X,Y)} \ (= \varphi^{-1}(\mathbb{R})$, where $\varphi: T \rightarrow T/M \cong \mathbb{C}$ is the canonical projection) is a $t$-local domain, since its maximal ideal $M=
(X,Y) \mathbb{C}[X,Y]_{(X,Y)}$ is divisorial as an ideal of $D$, being $M= (D:T)$.
Moreover, every nonzero prime ideal of $D$ is  a $t$-ideal. Indeed, for the well known properties of the pullback constructions, every nonzero nonmaximal prime ideal $P$ of $D$ is such that $P =Q \cap D$, where $Q$ is a nonzero nonmaximal prime ideal of $T$, and moreover $D_P$ is canonically isomorphic to $T_Q$  \cite[Theorem 1.4 (part (c) of the proof)]{[F]}. Since $T_Q$ is a DVR, $D_P$ is a DVR too and hence $PD_P$ is a $t$-ideal  of $D_P$ and, in particular, $P$ is a $t$-ideal of $D$.

Finally, $D$ is not $DT$-domain or, equivalently for Noetherianity, $D$ is not a divisorial domain, since $\dim(D) =2$ (Remark \ref{heinzer}).  Explicitly,  for instance, $M^2$ is not a divisorial  ideal (or, equivalently, not a $t$-ideal)  of $D$ (and of $T$), since
$(D:M^2) = ((D:M):M) =(T:M)=T$ and so $(D:(D:M^2))=(D:T) =M$.
}
\smallskip

\end{example}
%%%%%%%%%

Recall that an overring $T$ of an integral domain $D$ is is called {\it $t$-linked over $D$} if, for each nonzero finitely generated ideal $J $ of $D$ such that $J^t =D$, then $(JT)^t =T$. An integral domain is {\it $t$-linkative} if every overring is $t$-linked \cite{dobbs-et-al-1989}.

%%%%%%%%%%%% PROPOSITION
\begin{proposition} Let $D$ be an integral domain.
Then, $D$ is $t$-local domain
if and only if $D$ is a local $t$-linkative domain.
\end{proposition}
%%%%%%%%%%

The previous proposition is a straightforward consequence of the following theorem.

%%%% THEOREM {t-inv} 
\begin{theorem}\label{t-inv} {\emph{(Dobbbs-Houston-Lucas-Zafrullah, 1989 \cite[Theorem 2.6]{dobbs-et-al-1989})}} Let $D$ be an integral domain. The following are equivalent.
\begin{itemize}
\item[(i)] Every overring of $D$ is $t$-linked over $D$.
\item[(ii)] Every valuation overring of $D$ is $t$-linked over $D$.
\item[(iii)]  Every maximal ideal of $D$ is a $t$-ideal.
\item[(iv)]  For each nonzero proper ideal $I$ of $D$, $I^t \neq D$.
\item[(v)]  For each nonzero proper finitely generated ideal $J$ of $D$, $J^t \neq D$.
\item[(vi)]  Each $t$-invertible ideal of $D$ is invertible.
\end{itemize}
\end{theorem}
% %%%%%%%%%%%

\medskip

Finally, we introduce a construction for building new examples of $t$-local domains. 

We recall that, given an integral domain $D$,  the {\it Nagata ring of $D$} (see, for instance, \cite[Section 33]{[G]}) is defined as follows:

$$
D(X) := \{f/g \mid f, g\in D[X],\ g \neq 0, \mbox{ with } \co(g) = D\},
$$
(where $\co(h)$ is the content of a polynomial $h\in D[X]$). 

First in \cite{[K-89]} and then in \cite{[FL]}, the construction of the Nagata ring was extended to the case of an arbitrary chosen star (or, even semistar) operation. Given a star operation $\ast$ on $D$, set:
$$
\Na(D, \ast) := \{f/g \mid f, g\in D[X],\ g \neq 0, \mbox{ with } \co(g)^\ast = D\}.
$$
With this notation $\Na(D, d) = D(X)$. Moreover, it is clear that  
$$
\Na(D, v) = \Na(D, t) = \Na(D, w)
$$ 
since, for each nonzero finitely generated ideal $F$ of $D$,  $F^v = F^t$ and, moreover, $F^t =D$ if and only if $F^w=D$, because   $\Max^t(D)= \Max^w(D)$.

%%%%%%% PROPOSITION
\begin{proposition} \ Let $D$ be an integral domain. 
\begin{itemize}
\item[(1)]  The Nagata ring  $\emph{\Na}(D, v)$ is a $DW$-domain; in particular, if $\Max^t(D) =\{Q\}$ is a singleton, then  $\emph{\Na}(D, v)$ is a $t$-local-domain with maximal $t$-ideal $Q\emph{\Na}(D, v)$.
\item[(2)]
The following are equivalent.
\begin{itemize}
\item[(i)] $D$ is a $t$-local  domain.
\item[(ii)] ${\emph{\Na}}(D, v) =D(X)$ and $D(X)$ is local.
\item[(iii)] $D(X)$ is a $t$-local domain.
\end{itemize}
\end{itemize}
\end{proposition}
%%%%
\begin{proof}
(1) 
Recall that $\boldsymbol{\mathcal N} := \{ g\in D[X] \mid  g \neq 0 \mbox{ and } \co(g)^\ast = D\}$ is a saturated multiplicatively closed subset of $D[X]$, 
$\boldsymbol{\mathcal N} = D[X] \setminus \left( \bigcup\{QD[X] \mid Q \in \Max^{\stf}(D)\}\right)$, $\Na(D, v) = D[X]_{\boldsymbol{\mathcal N}}$, and $\Max(\Na(D, v)) =
 \{ Q{\Na(D, v)}  \mid Q \in  \Max^t(D)\}$ (see \cite[Proposition 3.1]{[FL]} or \cite[Proposition 2.1]{[K-89]}). Then, it is easy to see that $\Na(D, v) _{Q\Na(D, v)} $ $= D[X]_{QD[X]} =D_Q(X)$ and $Q\Na(D, v) = QD_Q(X) \cap \Na(D, v)$, for each $Q \in \Max^{t}(D)$, and so:
 $$
\Na(D, v) = \bigcap \{ D_Q(X) \mid Q \in \Max^t(D)\}.$$
 Moreover, for each ideal $I$ of $D$, $(I\Na(D,v))^t =I^t\Na(D,v)$ \cite[Corollary 2.3]{[K-89]}. Therefore, in particular,  $Q\Na(D, v)$ is a $t$-ideal 
of $\Na(D,v)$ for each  $Q\in \Max^t(D)$, i.e., $\Max(\Na(D, v)) = \Max^t(\Na(D, v))$.

(2) (i)$\Rightarrow$(ii). We already observed that $\Na(D, v) = \Na(D, t) = \Na(D, w)$. In the present situation $d = w$ and so $\Na(D, w) = \Na(D, d) =D(X)$.

(ii)$\Rightarrow$(iii).  Obvious, since we have shown in (1) that, when $D$ is $t$-local, $\Na(D, v)$ is    $t$-local too. 

(iii)$\Rightarrow$(i) Since the maximal ideals of $D(X)$ are exactly the ideals $M(X):= MD(X)$, with $M \in \Max(D)$ \cite[Proposition 33.1]{[G]}, and since $M(X)^t= M^t(X)$ \cite[Corollary 2.3]{[K-89]}, the conclusion is straightforward.
\end{proof}

By the previous proposition, the Nagata ring can be used to give new examples of $DW$-domains and, in particular, of $t$-local domains.
 For instance, it is known that $D(X)$ is treed  (i.e., the prime spectrum is a tree under the set theoretic inclusion $\subseteq$) if and only if $D$ is treed and the integral closure $\overline{D}$ of $D$ is a Pr\"ufer domain \cite[Theorem 2.10]{[ADF]}. 
 Thus, if we take a treed domain $D$ such that $\overline{D}$ is not Pr\"ufer, in this case $D(X)$ is a  $DW$-domain,  but not treed.
For an explicit example, take 
 $D:=\mathbb Q+ U\mathbb Q(V)[[U]]$, where $U$ and $V$ are two indeterminates, then $D= \overline{D}$   \cite[Remark 2.11]{[ADF]}, $D$ is a $t$-local (treed) integrally closed domain but not a valuation domain, and thus $D(X)$ is a $t$-local non treed integrally closed domain, since the integral closure $\overline{D(X)}=\overline{D}(X)=D(X)$   \cite[Proposition 2.6]{[ADF]}.

\medskip

%%%%%%%%%%%%%%%%%%%%%%%%%
%%%%%%% SECTION 4
%%%%%%%%%%%%%%%%%%%%%%%%%%%
%%%%%%%%%%%%%%%%%%%%%%%%%%%%%%%%%
\section{Comparable elements and $t$-local domains}
%%%%%%%%%%%%%%%%%%%%%%%%

A {\sl nonzero} element $c\in D$ is called {\it comparable in} $D$ if, for all $x\in D$,
we have $cD\subseteq xD $ or $xD\subseteq cD$. It is easy to see that $c\in D$ is comparable if $cD$ is comparable (under inclusion) with each ideal $I$ of $D$.
The following result is essentially Lemma 3.2 of \cite{[CDZ]}.

%%%%% LEMMA
\begin{lemma} Let $\alpha$ be a nonzero nonunit element of a local domain $(D,M)$. If,  for each $x\in D$, $\alpha D + xD =yD \subseteq M$, then $\alpha$ is a comparable element.
\end{lemma}
%%%%%%
\begin{proof} By the assumption, it follows that  $(\alpha/y) D + (x/y) D =D$ and, since $D$ is local, $\alpha/y $ or $x/y$ is a unit of $D$. Thus, the element $y$ is an  associate of $\alpha$ or of $x$. In the first case, $y | x$ (or, equivalently, $\alpha | x$) and, in the second case, $y | \alpha$ (or, equivalently, $ x | \alpha$). Therefore, $\alpha$ is a comparable element of $D$.
\end{proof}
%%%%%%%%%%%%%%%%%%%

%%%%% LEMMA\label{factor}
\begin{lemma} \label{factor} Let $c$ be a comparable element in an integral domain $D$. If $h$ is a nonunit factor of $c$, then $h$ is also a comparable element of $D$.
\end{lemma}
%%%%
\begin{proof} Let $c= hy$ and let $x \in D$.  Then $cD + xyD = hyD+xyD= y(hD+xD)$ coincides with $cD$ or $xyD$, since $c$ is comparable. In the first case, 
$y(hD+xD)=cD= yhD$, thus $hD +xD = hD$, i.e.,  $x|h$. In the second case, $y(hD+xD) =xyD$ and thus $hD+xD =xD$, i.e., $h|x$.
\end{proof}
%%%%%%%

The comparable elements were introduced and studied in \cite{[AZ]} to prove,  in case of valuation domains, a Kaplansky-type
theorem (recall that Kaplansky proved that  an integral domain $D$ is a UFD if and only if every
nonzero prime ideal of $D$  contains a prime element   \cite[Theorem 5]{[Kap]}).

%%%%%%% LEMMA (KAPLANSKY TYPE THEOREM)
\begin{lemma}\label{Kap}{\em{(D.D. Anderson and M. Zafrullah \cite[Theorem 3]{[AZ]})}} 
An integral domain $D$ is a valuation domain if and only if every
nonzero prime ideal of $D$ contains a comparable element. 
\end{lemma}
%%%%%%%%%

An important part
of the result was the proof of the fact that the set of all comparable
elements of $D$ is a saturated multiplicative set.

We recall in the next lemma some of the consequences of the existence of a nonzero nonunit comparable element in an integral domain.

%%%%%%%%%%%%% LEMMA  \label{C} Gilmer-Heinzer-Zafrullah
\begin{lemma}\label{C} {\em{(Gilmer-Mott-Zafrullah \cite[Theorem 2.3]{[GMZ]})} }
Suppose the integral domain $D$
contains a nonzero nonunit comparable element and let $\mathscr{C}$ be the (nonempty) set of
nonzero comparable elements of $D$. Then:
\begin{itemize}
\item[(1)] $P:=\bigcap \{cD \mid c\in \mathscr{C}\}$ is a prime ideal of $D$ and $D\setminus  P=\mathscr{C}$ (in particular, $\mathscr{C}$ is a saturated multiplicative set of $D$).

\item[(2)] $D/P$ is a valuation domain.

\item[(3)] $P=PD_{P}$.

\item[(4)] $D$ is  local, $P$ compares with every other ideal of $D$
under inclusion, and 
$\dim(D)= \dim(D/P) + \dim(D_{P})$.

\item[(5)]  If $T$ is any integral domain such that there is a nonmaximal
prime ideal $Q$ of $T$ such that {\em (a)} $T/Q$ is a valuation domain, and {\em (b)} 
$Q=QT_{Q}$, then each element of $T\setminus Q$ is comparable.

\item[(6)]
 If, in
addition, $Q$ is minimal in $T$ with respect to properties {\em (5, a)} and {\em (5, b)} above, then 
$T \setminus Q$ is   precisely the set of nonzero comparable elements of $T$.
\end{itemize}
\end{lemma}
%%%%%%%%%i%%%%%%%%%%%%%%%%%%
 
Of course, an integral domain $D$ is a valuation domain if and only if every nonzero element of 
$D$ is comparable. As an easy consequence of the previous lemma we obtain immediately the following.

%%%%%% COROLLARY  \label{CA}  
\begin{corollary} \label{CA}  Suppose the integral domain $D$
contains a nonzero nonunit comparable element and let $\mathscr{C}$ be the (nonempty) set of
nonzero comparable elements of $D$. 
Then, $D$ is a valuation domain if and
only if $\cap \{cD \mid c\in \mathscr{C}\}=(0)$.
\end{corollary}
%%%%%%%%%%%%
\begin{proof} 
The statement follows from (1) and (2) of Lemma \ref{C}.
\end{proof}
%%%%%%%%

Recall that E.D. Davis proved that, given a ring $S$  and a subring $R$, if $R$ is local then $(R,S)$ is a normal pair (i.e., every ring $T$, $R\subseteq T \subseteq S$, is integrally closed in $S$) if and only if there is a prime ideal $Q$ in $R$ such that $S =R_Q$, $Q = QR_Q$, and $R/Q$ is a valuation domain 
\cite [Theorem 1]{[D]}. From the previous remark and Lemma \ref{C}, we deduce immediately the following.

%%%%% COROLLARY
\begin{corollary} Suppose the integral domain $D$
contains a nonzero nonunit comparable element. Let $\mathscr{C}$ be the   set of
nonzero comparable elements of $D$ and $P:=\bigcap \{cD \mid c\in \mathscr{C}\}$, as in  Lemma \ref{C}(1). In this situation, 
$(D,D_P)$ is a normal pair.
\end{corollary}
%%%%%%

\medskip

 In \cite{[GMZ]},  a part of the
following result was proved as a consequence of Lemma \ref{C}.
We next prove, directly, that the existence of a nonzero nonunit comparable element in an integral domain is a sufficient but not necessary condition for being a $t$-local domain. 

%%%%%%%%%%%%%%%  PROPOSITION  \label{B}   [GMZ, Corollary 2.5]).
\begin{proposition}  \label{B}
An integral domain $D$ that contains
a nonzero nonunit comparable element is a $t$-local domain, while a $t$-local domain may not contain a nonzero nonunit comparable element.
\end{proposition}
%%%%%%%%%
\begin{proof}
Let $D$ be an integral domain and let $c$ be a nonzero nonunit comparable  
element in $D$. 
We first show that $D$ is local. Suppose, by way of
contradiction, that there exist two co-maximal nonunit elements $x,y$ in $D$,
i.e., $rx+sy=1$ for some $r,s\in D$. Now, as $c$ is comparable, $c|rx$ \ or \ 
$rx|c$. So $rx$ has a nonzero nonunit comparable factor $c$ or, being a factor of 
$c$, $rx $ is a nonzero nonunit comparable element. 
Thus, in both cases, $rx$ has a nonzero nonunit
comparable factor $h$.
 Similarly $sy$ has a nonzero nonunit comparable factor $k$.
 Since $h,k$
are comparable, $h|k$ or $k|h$,  say $h|k$. Thus, assuming that $rx+sy=1$, we
get the contradictory conclusion that a nonunit divides a unit. So, $D$ is
local. We denote by $M$ its maximal ideal.

Next,  let $x_{1},x_{2},\dots,x_{n}\in
M $ and note that, as above, each of the $x_{i}$ has a nonzero nonunit comparable
factor $h_{i}$.
 Thus, $(x_{1},x_{2},\dots,x_{n}) \subseteq (h_{1},h_{2},\dots,h_{n})$. 
 
 Now, consider $h_{1}, h_{2}$.   They must  have  a nonzero nonunit
common factor $k_{1}$ (which is equal to $ h_{1}$ or $h_{2})$.
So, $(x_{1},x_{2},\dots, x_{n})$ $
\subseteq (h_{1},h_{2},\dots, h_{n})\subseteq (k_{1}, h_{3}, \dots, h_{n})$. 
Continuing
this process,  we eventually get a nonzero nonunit comparable element $k$ such that 
$(x_{1},x_{2},\dots, x_{n})$ $\subseteq (h_{1},h_{2},\dots, h_{n})\subseteq (k)\subseteq M$.
But,  as $(x_{1},x_{2},\dots, x_{n})\subseteq (k)$ implies $(x_{1},x_{2},\dots, x_{n})^{v}\subseteq (k)$, we conclude that, for each finitely generated ideal 
$(x_{1},x_{2},\dots, x_{n})\subseteq M,$ $(x_{1},x_{2},\dots, x_{n})^{v}\subseteq M$.
Thus, $D$ is a $t$-local domain. 

For the converse, note that a one dimensional
 local domain has only one nonzero prime (=maximal) ideal and so it is a valuation
ring if and only if it contains a nonunit comparable element, by the
Kaplansky-type theorem mentioned above (Lemma \ref{Kap}). The proof is complete once we note
that there do exist one-dimensio\-nal, (Noetherian $t$-)local domains that are
not valuation domains (in fact, non integrally closed domains) (e.g., $\mathbb{R} + X \mathbb{C}[\![X]\!]$).

Note also  that there even exist  1-dimensional $t$-local integrally closed domains that are not valuation domains (e.g., $\overline{\mathbb{Q}} + X \mathbb{C}[\![X]\!]$, where $\overline{\mathbb{Q}}$ is the algebraic closure of ${\mathbb{Q}}$ in $ {\mathbb{C}}$). 
\end{proof}
%%%%%%%%%%%%%%

%%%%%%%%%%  REMARK
\begin{remark}
\em{Note that the previous example shows that a local domain with divisorial maximal ideal may not contain a nonzero nonunit comparable element. On the other hand, a valuation domain $V$ with  nonprincipal maximal ideal (in particular, $\dim(V) \geq 2$) is a domain containing
a nonzero nonunit comparable element and so it is a $t$-local domain with nondivisorial maximal ideal.
}
\end{remark}
%%%%%%%%%

\medskip
Recall that an integral domain $D$ with quotient field $K$ is called a {\it pseudo-valuation domain} (for short, {\it PVD}) if $D$ is local and  the  maximal ideal $M$ of $D$ is strongly prime (i.e., 
whenever elements $x$ and $y$ of $K$ satisfy $xy \in M$, then either $x \in M$   or $y \in M$).  From the proof of the previous Proposition \ref{B}, we give now  a general class of $t$-local domains that do not contain nonzero nonunit comparable elements.

% %%%%%%%% EXAMPLE
\begin{example}
\em{ Let $(T, M)$ be any local domain, let $\boldsymbol{k}(T) := T/M$, let $\varphi: T \rightarrow \boldsymbol{k}(T)$ be the canonical projection, and let $F$ be a proper subfied of $\boldsymbol{k}(T)$.  Set $D:=\varphi^{-1}(F)$. 
It is known that $D$ is   a local domain with  maximal ideal $M$ and $(M:M) =(D:M) =T$. Since $M=(D:T)$, it is easy to see  that $M$ is a divisorial ideal in $D$ and, in particular, a $t$-ideal. Thus, $(D,M)$ is a $t$-local domain.
In particular,  any PVD  is a $t$-local domain \cite[Theorem 2.10]{[HH-PVD]}. 
}
\end{example}
%%%%%%%%

%%%%%%% REMARK
\begin{remark}
\em{
Note that the argument used in the previous example can be used to construct a more general class of $t$-local domains. 
Start from a (not necessarily local) integral domain $T$ such that its Jacobson ideal $J(T)$ is nonzero and suppose that the ring $T/J(T)$ contains properly a field $F$.  Let $\varphi: T \rightarrow T/J(T)$ be the canonical projection and let $D := \varphi^{-1}(F)$, then $D$ is a $t$-local domain.
}
\end{remark}
%%%%%%%%%%%%%%%%

\medskip

A fractional ideal $E\in \F(D)$ is said to be $v$-{\it{invertible}} (respectively, $t$-{\it{invertible}}) if there
is $G\in \F(D)$ such that ($(EG)^{v}=D$ (respectively, $(EG)^{t}=D$). Obviously, every
invertible ideal is $t$-invertible.

  Recall that a {\it GCD domain} is  an integral domain $D$ such that,  for each $a, b \in D$, $aD \cap bD$ is principal or, equivalently, $(a,b)^v$ is principal. Therefore, a GCD domain (e.g., a B\'ezout domain) is a P$v$MD.

%%%%% COROLLARY   \label{BA}
\begin{corollary}\label{BA} 
Let $D$ be a  P$v$MD, not a field. Then,  $D$ is a valuation domain if and only if $D$ contains a nonzero nonunit comparable element.
\end{corollary}
%%%%%%%%%
\begin{proof}
The statement follows from Proposition \ref{B}, from the fact that a $t$-local P$v$MD is a valuation domain anyway and from the fact that
a valuation domain that is not a field must contain many nonunit comparable
elements (in fact, all nonunit elements are comparable).
\end{proof}
%%%%%%%%

From the previous corollary it follows that every Krull domain (e.g., UFD) containing a nonzero nonunit comparable element is a DVR and that every GCD domain  containing a nonzero nonunit comparable element is a valuation domain.

 Now, here
comes something more general and  a tad surprising. Call an integral domain $D$ {\it atomic} if
every nonzero   nonunit of $D$ is expressible as a finite product irreducible
elements.  An irreducible element is called also   {\it atom}. For instance, every Noetherian domain and every UFD is atomic.

%%%%%%% COROLLARY  \label{BB}
\begin{corollary} \label{BB}
An atomic domain that contains a nonzero nonunit comparable element
is a DVR.
\end{corollary}
%%%%%%%%
\begin{proof}
 Let $D$ be an atomic domain and let $c$ be a nonzero nonunit comparable
element in $D$. Then, by Proposition \ref{B}, $D$ is $t$-local domain; denote by $M$ its maximal ideal.
 Let $h$ be an irreducible factor of $c$. Then $h$ is a comparable
element, being a factor of a comparable element (Lemma \ref{factor}). 
 So, for every $x$ in $D$,  either $h|x$ or $x|h$. 
 Now, as $h$ is irreducible, $x|h$ means that $x$ is a unit or $x=h$. Thus, for all nonunits $x\in D$, necessarily  $h|x$. That is $M=hD$ and so $h$
is a prime element in $D$. 
Next, as $h|x$ for each nonzero nonunit $x\in D$,  we
have $x=x_{1}h$ and if $x_{1}$ is a nonunit then $x_{1}=x_{2}h$ and so $
x=h^{2}x_{2}$. Continuing this way, 
since $D$ is
atomic, for each nonzero nonunit $x\in D$ there is an integer $n=n(x)$ (depending on $x$) such
that $x=h^{n}x_{n}$ where $x_{n}$ is a unit. But then  we can conclude that $D$ is a DVR and $h$ is a uniformizing parameter of $D$.
\end{proof}
%%%%%%%%%%%%

Corollary \ref{BB} was first proved  for Noetherian domains; we thank Tiberiu Dumitrescu  for suggesting the atomic domain assumption. 
 With hindsight we
can prove  a more precise result.

%%%%%%% COROLLARY  \label{BD}
\begin{corollary} \label{BD}
 Let $D$ be a domain that contains a nonzero nonunit comparable
element. 
\begin{enumerate}
\item[(1)]
In this situation, $D$ is local (Proposition \ref{B}) and
the maximal ideal of $D$ is generated by the nonunit comparable elements of $D$.

\item[(2)] The integral domain $D$ contains an atom $\alpha$ if and only if $\alpha$ is the generator
of the (unique) maximal ideal of $D$ and, hence,  $\alpha$ is a prime and comparable element.
\end{enumerate}
\end{corollary}
%%%%%%
\begin{proof}  (1) By Proposition \ref{B}, $D$ is $t$-local; let $M$ denote the maximal ideal of $D$. 
With the notation of Lemma \ref{C}, $M$ properly contains the comparable prime ideal $P$ of $D$.
 If $(x_{1},x_{2},\dots, x_{n})$ is a finitely generated ideal and  
$P \subseteq (x_{1},x_{2},\dots, x_{n})\subseteq M$, since $D/P$ is a valuation domain, then  $(x_{1},x_{2},\dots, x_{n}) =(x)$ for some $x \in \{x_{1},x_{2},\dots, x_{n}\}$. Therefore, since $M= M^t$, $M$ is generated by the nonunit comparable elements of $D$.

(2)  Let $\alpha$ be an atom of $D$ and let 
$c$ be a nonzero nonunit comparable element of $D$.
Then, either $c|\alpha$ or $\alpha|c$. If $c|\alpha$
then, as $\alpha$ is an atom and $c$ a nonunit, $c$ and $\alpha$ must be associate,
so $\alpha$ is a comparable element. 
If, on the other hand, $\alpha|c$ then $\alpha$ is a
comparable element, being a factor of a comparable element (Lemma \ref{factor}). Thus, as above, $\alpha D=M$.

The converse is obvious, indeed if the maximal ideal $M$ of a local domain $D $ is principal and $M=\alpha D$ then, up to associates, $\alpha$ is the only atom in $
D$.
\end{proof}
%%%%%%%%%%%%%%

Note that if, instead considering atoms (=irreducible elements), we consider prime elements, we can state a result analogous to the previous corollary in a more general setting, with a different proof.

%%%%%%% PROPOSITION  \label{prime}
\begin{proposition}\label{prime} 
 Let $D$ be a domain.
 
 \begin{enumerate}
 \item[(1)] If    a maximal $t$-ideal  $M$ of $D$ contains a prime element $p$, then $M = pD$.
 
\item[(2)]  If $(D,M)$  is a $t$-local domain (e.g., if $D$ contains a nonzero nonunit comparable
element), then  $D$ contains a prime element $p$ if and only if $p$ is the generator
of the  maximal ideal of $D$ and, hence,  $p$ is a comparable element.
\end{enumerate}
\end{proposition} 
%%%%%%
\begin{proof} (1)
Let  $p$ be a prime element of a domain $D$ then, for each $x$ in $D$,
$pD\cap xD= xD$ 
 or $pD\cap xD= pxD$.

So, 
$$((p,x)D)^{-1}=\frac{pD\cap xD}{
px}=\left(\frac{1}{p}\right)D \quad \mbox{  or  } \quad  ((p,x)D)^{-1}=D.$$
 But then $((p,x)D)^{v}=pD$ or $((p,x)D)^{v}=D$. 
 So, if a prime
element $p$ belongs to a maximal $t$-ideal $M$ then $M=pD$.

(2)  If a prime
element $p$ belongs to a $t$-local ring $(D, M)$ then $M=pD$, by (1) and consequently $p$
is a comparable element of $D$.
\end{proof}
%%%%%%%%%%%%%%%%%%%%%

 It is well known that, if $p$ is a prime
element in an integral domain $D$, then $\bigcap_{n \geq 0} p^{n}D$ is a prime ideal too (see, for instance, Kaplansky \cite[Exercise 5, pages 7-8]{[Kap]}).

\medskip

%%%%%%%%%%%%%%%%%%%  THEOREM \label{D} 
\begin{theorem} \label{D}   
If a domain $D$ contains a nonzero nonunit comparable element
then, for every nonzero nonunit comparable element $x$ of $D$, we have that $Q:=\bigcap_{n \geq 0} x^{n}D$
is a prime ideal such that $D/Q$ is a valuation domain and $Q=QD_{Q}$.

Conversely, if there is a nonzero element $x$ in a domain $D$ such that $Q:=\bigcap_{n \geq 0} x^{n}D$  is a prime ideal, $D/Q$ is a valuation domain, and $
Q=QD_{Q}$, then $D$ is $t$-local and $x$ is a comparable element of $D$.
\end{theorem}
%%%%%%%%%%%
\begin{proof} Indeed $Q$ is an ideal, being an intersection of ideals. 
Now, consider 
$S: =D\backslash Q$ and let $a,b\in S$. Then $a\notin x^{m}D$ for some
positive integer $m$ and $b\notin x^{n}D$ for some positive integer $n$.
Since $x$ and hence $x^{m},\ x^{n}$ are comparable, we conclude that 
$aD\varsupsetneq x^{m}D$ and $bD\varsupsetneq x^{n}D$.
Therefore, $abD\varsupsetneq ax^{n}D \varsupsetneq x^{n+m}D$ and so $ab\in S$ and $Q$ is a prime ideal.

From the above proof it follows that $S$ consists of factors of powers of
the comparable element $x$ and so every element of $S$ is comparable; this
implies that $D/Q$ is a valuation domain. 
Next,  let $\alpha/\tau \in QD_{Q}$ where $\alpha\in Q$
and $\tau \in D\backslash Q$. In particular,  $\tau$ divides some power of $x$ and so  $\tau$ is comparable. Hence, 
$\alpha D \subseteq Q \varsubsetneq \tau D$ which means that for some nonunit $y$ we have $\alpha=\tau y$.
As $\tau\notin Q$, then necessarily  $y\in Q$.
So $\alpha/\tau =y\in Q$. Thus $QD_{Q}\subseteq Q$, i.e. $Q =QD_{Q}$. 

The
converse follows from  Lemma \ref{C}(5) and Proposition \ref{B}  (see also \cite[Theo\-rem 2.3]{[GMZ]}).
\end{proof}
%%%%%%%%%

Note that there are integral domains that may or may not be local, but
have elements $x$ such that $\cap x^{n}D=:Q$ is a prime ideal such that $
Q=QD_{Q},$ but $D/Q$ is not a valuation domain. Here are some examples using
the $D+M$ construction studied by  Gilmer   \cite[page 202]{[G]}.

 We start from   a valuation domain $V$, with quotient field $K$,
expressible as $V=\boldsymbol{k}+M$, where $\boldsymbol{k}$ is a subfield of $V$ (and $K$) and $M$ is the maximal
ideal of $V$; thus, in the present situation,  the residue field $V/M$ is canonically isomorphic to $\boldsymbol{k}$.
 Let $D$ be a subring of $\boldsymbol{k}$. The ring $R:=D+M$ (subring of $V$) with  quotient field $K$ (the same as $V$)  has some interesting properties due to the
mode of this construction, as indicated for instance in \cite{[BG]}  (see also \cite[Theorem 1.4]{[F]}).
Our concrete model for these
examples would be  $V:=\boldsymbol{k}[\![X]\!]=\boldsymbol{k}+X\boldsymbol{k}[\![X]\!]$.

%%%%%%%% EXAMPLE \label{DA} 
\begin{example} \label{DA} 
\em{ 
Given a field $\boldsymbol{k}$, let $D$ be a 1-dimensional local domain  contained in $\boldsymbol{k}$, with quotient
field $F \ (\subseteq \boldsymbol{k})$ and suppose that $D$ is not a valuation domain.
Then $R:=D+X\boldsymbol{k}[\![X]\!]$ is a (local) 2-dimensional domain such that, for each nonzero nonunit $x$ in $D$, we have 
$\bigcap_{n \geq 0} x^{n}R=X\boldsymbol{k}[\![X]\!]$. Indeed, for a  nonunit $x$ in a
1-dimensional local domain $D$, we have $\bigcap_{n \geq 0} x^{n}D=(0)$  and so $\bigcap_{n \geq 0}   x^{n}R=X\boldsymbol{k}[\![X]\!]$.
Moreover, since $R_{X\boldsymbol{k}[\![X]\!]} = F+ X\boldsymbol{k}[X]\!]$, then 
$X\boldsymbol{k}[\![X]\!]R_{X\boldsymbol{k}[\![X]\!]}=X\boldsymbol{k}[\![X]\!](F+X\boldsymbol{k}[X]\!])=X\boldsymbol{k}[\![X]\!]$.  In this situation, $R/X\boldsymbol{k}[\![X]\!]=D$.
}
\end{example}
%%%%%%%%%%%%%%

What makes the above example work is the fact that, for a nonunit $x$ in a
one dimensional  local domain $D$, we have $\bigcap_{n \geq 0} x^{n}D =(0)$. Call an
integral domain $D$ an {\it Archimedean domain} if, for all nonunit elements $x$
in $D$, we have $\bigcap_{n \geq 0} x^{n}D=(0)$  \cite[Definition 3.6]{[Sh]} (this class of domains was previously considered in \cite{[O]} without naming them).
By the Krull intersection theorem, every Noetherian domain is Archimedean. Since Mori domains satisfy the ascending chain condition on principal ideals, they are Archimedean; in particular,  Krull domains are Archimedean. The class of Archimedean domains includes also completely integrally closed domains \cite[Corollary 5]{[GH]} and 1-dimensional integral domains   \cite[Corollary 1.4]{[O]}.

An Archimedean  (possibly non local or any dimensional) version of the previous Example \ref{DA} is given next.

%%%%%%%% EXAMPLE \label{DB} 
\begin{example} \label{DB}  
\em{ 
Given a field $\boldsymbol{k}$, let $D$ be an Archimedean domain   contained in $\boldsymbol{k}$, with quotient
field $F \ (\subseteq \boldsymbol{k})$ and suppose that $D$ is not a valuation domain. Then, as above, 
 $ R:=D+X\boldsymbol{k}[\![X]\!]$ is such that, for each nonzero nonunit $x$ in $D$, we have $\bigcap_{n \geq 0} x^{n}R=X\boldsymbol{k}[\![X]\!]$, 
$X\boldsymbol{k}[\![X]\!]=X\boldsymbol{k}[\![X]\!]R_{X\boldsymbol{k}[\![X]\!]}$ and $R/X\boldsymbol{k}[\![X]\!]=D$. In the present situation,  $\Max(R)$ has the same cardinality of  $\Max(D)$ and $\dim(R) = \dim(D)+1$.
}
\end{example}

%%%%%%%% EXAMPLE \label{DC} 
\begin{example} \label{DC}  
\emph{ Let $D$ be an integral domain and $S$ a multiplicative subset of $D$. Following the construction $R:=D+XD_{S}[X]$ of \cite{[CMZ]}, if $s$ is a
nonunit element in $S$ such that $\bigcap_{n \geq 0} s^{n}D=(0)$ then $\bigcap_{n \geq 0}
s^{n}R=XD_{S}[X]$ a prime ideal of $R$. Also in this case  $R/XD_{S}[X]=D$, which might  not be a valuation domain. However, in the present situation,
$XD_{S}[X] \subsetneq XD_{S}[X](R_{XD_{S}[X]}) = XD_S[X]_{(X)}$.
}
\end{example}
%%%%%%%%%%%%%%%%%%%%%%%

\bigskip

%%%%%%%%%%%%%%%%%%%%%%%%%%%
%%%%%%%%%%%% SECTION  5
%%%%%%%%%%%%%%%%%%%%%%%%%%%
\section{From $t$-local domains to valuation domains}
%%%%%%%%%%%%%%%%%%%%%%%%%%%%%%

Because in a valuation domain $(V,M)$ every finitely generated ideal is
principal, the maximal ideal $M$ is obviously a $t$-ideal. So $t$-local
domains are ``cousins'' of valua\-tion domains,  but sort of far removed. 
For instance, a localization of a $t$-local domain is not necessarily $t$-local   (see, for instance,  Example \ref{bad behaved 2} or \cite{[Z-WB]}),  but of course a localization of a valuation domain is a valuation domain.

Explicitly,   a more simple
example is given by  $R:=\mathbb{Z}_{(p)}+(X,Y)\mathbb{Q}[\![X,Y]\!]$. The integral domain $R$ is local with maximal ideal $M:=p\mathbb{Z}_{(p)}+(X,Y)\mathbb{Q}[\![X,Y]\!]= pR$,
and so it is obviously  a $t$-local domain. However,  $R[1/p]=R_Q= \mathbb{Q}[\![X,Y]\!]$, where $Q:=(X,Y)\mathbb{Q}[\![X,Y]\!]$,  is  a 2-dimensional local  Noetherian Krull domain, and so it is
far away from being $t$-local.

So it is legitimate to ask:
Under what conditions is a $t$-local domain a valuation domain?
Here we address this question.

 The following is a simple result that hinges
on the fact that if $F$ is a nonzero finitely generated ideal in a $t$-ideal $I$
then $F^{v}\subseteq I$.

%%%%%%%%%%%%%%% PROPOSITION \label{E}
\begin{proposition} \label{E} 
For a finite set of elements $x_{1}, x_{2}, \dots, x_{n},$ in a $t$-local
domain $(D,M)$, the following are equivalent.
\begin{itemize}

\item[(i)] $(x_{1},x_{2}, \dots, x_{n})^{v}=D$.

\item[(ii)]  At least one $x_{i}$ is a unit.

\item[(iii)]  $(x_{1},x_{2}, \dots, x_{n})=D$.
\end{itemize}
\end{proposition}
%%%%%%%%%%%%%%%
\begin{proof}
Clearly, (ii) $\Rightarrow$ (iii) $\Rightarrow$ (i).

(i) $\Rightarrow$ (ii) By the previous observation  $(x_{1}, x_{2}, \dots, x_{n}) \not\subseteq M$, and so at least one $x_{i} \notin M$.
\end{proof}
%%%%%%%%%%

%%%%%%%%%%%%%%%%% PROPOSITION \label{F} 
\begin{proposition} \label{F}
For an integral domain $D$ the following are
equivalent.
 
\begin{itemize}

\item[(i)]  $D$ is a valuation domain

\item[(ii)] $D$ is a  $t$-local GCD domain  (or, equivalently, a  $t$-local  B\'ezout domain).

\item[(iii)] $D$ is a $t$-local  P$v$MD (or, equivalently, a $t$-local  Pr\"ufer domain).
\end{itemize}
\end{proposition}
%%%%%%%%%%
\begin{proof}
(i) $\Rightarrow $ (ii) $\Rightarrow $ (iii) are straightforward.

For (iii) $\Rightarrow $ (i) note for instance that, in a P$v$MD, every nonzero finitely
generated ideal $(x_{1},x_{2},...,x_{n})$ is $t$-invertible. But, by \cite[Proposition 1.12(1)]{[ACZ]}, $(x_{1},x_{2},...,x_{n})$ is a principal ideal.
\end{proof}
%%%%%%%%%%%%

Recall that a ring is {\it coherent} if every finitely generated ideal is finitely presented. It is well known that a commutative integral domain $D$ is  coherent if and
only if the intersection of every pair of finitely generated ideals is finitely generated   \cite[Theorem 2.2]{[Ch]}.

 Call a domain $D$ {\it a finite conductor domain} (for short, {\it FC domain}; this name was used for the first time in \cite{[Z-FC]})
if the intersection of every pair of principal ideals of $D$ is finitely generated.
Indeed,   ``finite conductor domain'' is a generalization of ``coherent
domain''.

%%%%%%%%%%%%%%%% PROPOSITION \label{FA}
\begin{proposition} \label{FA} 
For an integral domain $D$ the following are
equivalent.
\begin{enumerate}
\item[(i)] $D$ is a valuation domain.

\item[(ii)] $D$ is  an integrally closed coherent $t$-local domain.

\item[(iii)] $D$ is an integrally closed  finite conductor $t$-local domain.
\end{enumerate}
\end{proposition}
%%%%%%%%%%%%%%%%%%%%%
\begin{proof}
(i) $\Rightarrow $ (ii) $\Rightarrow $ (iii) are all straightforward. 

For
(iii) $\Rightarrow $ (i) note that an integrally closed FC domain is a P$v$MD
 \cite[Theorem 2]{[Z-FC]}  (or, \cite[Exercise 21, page 432]{[G]}) and we already observed that a $t$-local P$v$MD is a valuation domain (Proposition \ref{F}((iii)$\Rightarrow$(i))).
\end{proof}
%%%%%%%%%%

As an application of the previous proposition, we easily   obtain the following result due to S. McAdam.

%%%%%%%%%%%%%%%%%%%% COROLLARY  \label{FB} 
\begin{corollary} \label{FB} {\em{(S. McAdam \cite[Theorem 1]{[Mc]})}} 
Let $D$ be an integrally closed  local
domain whose primes
are linearly ordered by inclusion. Assume that $D$ is a FC domain, then $D$ is a valuation domain.
\end{corollary}
%%%%%%%%%%%%
\begin{proof}
By Proposition \ref{AD}, $D$ is $t$-local.
The conclusion follows from Proposition \ref{FA}((iii)$\Rightarrow$(i)).
\end{proof}
%%%%%%%%%%%%%%%%%%%

A nonzero element $r$ of a domain $D$ is called a {\it primal element} if for all $x,y\in
D\backslash \{0\}$ $r|xy$ implies that $r=st$ where $s|x$ and $t|y$. 
A
domain whose nonzero elements are all primal is called a {\it pre-Schreier domain}. 
An
integrally closed pre-Schreier domain was called a {\it Schreier domain} by P.M. Cohn in
his paper \cite[page 254]{[C]}. 
There, he showed that a GCD domain is a Schreier domain \cite[Theorem 2.4]{[C]}.

Based on considerations initiated by McAdam and Rush \cite{[McR]}, a module $M$ is said to be {\it  locally cyclic} if every finitely generated
submodule of $M$ is contained in a cyclic submodule of $M$.
 Thus, in particular,  an ideal $I$
of $D$ is locally cyclic if, for any finite set of elements $
x_{1},x_{2}, \dots, x_{n}\in I$, there is an element $d\in I$ such that $d|x_{k}$ for each $k$, $1\leq k\leq n$.

In \cite[Theorem 1.1]{[Z-PS]}, M. Zafrullah has shown that  {\it an integral domain $D$ is pre-Schreier if and only if for
all $a,b\in D\backslash (0)$ and $x_{1},x_{2}, \dots , x_{n}\in (a)\cap (b)$ there
is $d\in (a)\cap (b)$ such that $d|x_{k}$, for each $k$, $1\leq k\leq n$. }

Based on this,  we easily obtain the following.

 %%%%%%%%%%%%%%%%  LEMMA \label{FC}
 \begin{lemma}\label{FC}
 If $D$ is a pre-Schreier domain and 
$a,b\in D\backslash \{(0)\}$, then the following are equivalent:

\begin{enumerate}
\item[(i)]  $(a)\cap (b)$ is principal.

\item[(ii)] $(a)\cap (b)$ is finitely generated.

 \item[(iii)]
$(a)\cap (b)$ is a $v$-ideal of finite type.
\end{enumerate}
\end{lemma}
%%%%%%%%%%%%%%%%%%%%%
\begin{proof}
 Indeed (i) $\Rightarrow $ (ii) $\Rightarrow $ (iii) are all
straightforward. All we need is show (iii) $\Rightarrow $ (i). For this note
that if $(a)\cap (b)=\left( x_{1},x_{2}, \dots, x_{n}\right) ^{v},$ then,
$x_{1},x_{2}, \dots,x_{n}\in (a)\cap (b)$. 
Since $D$ is pre-Schreier, there is an element  
$d\in (a)\cap (b)$ such that $d|x_{k}$, for each $k$, $1\leq k\leq n$, i.e.,  $(x_{1},x_{2},\dots,x_{n})\subseteq (d)$. 
But then $\left( x_{1},x_{2},...x_{n}\right) ^{v}\subseteq
(d)$, and so $(d)\subseteq (a)\cap (b)=\left(
x_{1},x_{2},...x_{n}\right) ^{v}\subseteq (d)$.
\end{proof}
%%%%%%%%%%%%

Call a domain $D$ a {\it $v$-finite conductor} (for short, {\it $v$-FC}) {\it domain} if, for each pair 
$0\neq a,b\in D$, the ideal $(a)\cap (b)$ is a $v$-ideal of finite type. Then, recalling that a GCD domain is integrally closed, from Lemma \ref{FC},
we easily deduce: 

%%%%%%%%%%% COROLLARY  \label{GCD}
\begin{corollary}\label{GCD}
Let $D$ be an integral domain.
The following are equivalent.

\begin{enumerate}
\item[(i)]
$D$ is a GCD domain
\item[(ii)] $D$ is a
Schreier and a $v$-FC domain.
\item[(iii)] $D$ is a
pre-Schreier and a $v$-FC domain.
\end{enumerate}
\end{corollary}
%%%%%%%%%%%%%%%%

 With this preparation, we have the
following result.

%%%%%%%%%%% COROLLARY \label{FD}
\begin{corollary}  \label{FD} 
For   an integral domain $D$, the following are
equivalent:
\begin{enumerate}
\item[(i)] 
 $D$ is a valuation domain,

\item[(ii)]  $D$ is a pre-Schreier $t$-local  coherent domain,

\item[(iii)]  $D$ is a pre-Schreier $t$-local  FC domain,

\item[(iv)]  $D$ is a pre-Schreier $t$-local $v$-FC domain,

\item[(v)]   $D$ is a GCD $t$-local domain.
\end{enumerate}
\end{corollary}
%%%%%%%%%%%%%%%%%%%%%
\begin{proof} It is obvious that  (i) $\Rightarrow $ (ii) $\Rightarrow $ (iii) $\Rightarrow $ (iv);   (iv) $\Leftrightarrow $ (v)  by Corollary \ref{GCD}  and
(v) $\Leftrightarrow $ (i) by Proposition \ref{F}.
\end{proof}
%%%%%%%%%%%%%%%
\smallskip

Obviously, the above are not the only situations in which a  $t$-local integral domain becomes a
valuation domain. We describe next another interesting  situation of this phenomenon, in case of existence of a comparable element.

%%%%%%%  PROPOSITION \label {G} 
\begin{proposition}
\label {G} 
Suppose that an integral domain  $D$ contains a nonzero nonunit comparable element $x$
and let $Q:=\bigcap_{n\geq 0} x^{n}D$. Then, $D$ is a valuation domain if and only if 
$D_{Q}$ is a valuation domain.
\end{proposition}
%%%%%%%%%%%%%%

\begin{proof} Indeed, if $D$ is a valuation domain, since $Q$ is a prime ideal (Theorem \ref{D}), 
$D_{Q}$ is also a valuation domain and so we have only to take care of its
converse. 

The presence of a nonzero nonunit comparable element makes $D$ a 
$t$-local domain (Proposition \ref{B}). 
In order to prove that $D$ is a valuation domains, we consider the finitely generated ideals of $D$.
We split the proper finitely generated ideals into two
types: {\bf (a)} ones that contain a nonunit factor of a power of $x$ and {\bf (b)}
ones that do not contain a nonunit factor of a power of $x$.

 Ones in part
{\bf (a)} are principal by \cite[Theorem 2.4]{[GMZ]} and ones in part {\bf (b)} are contained in $Q$ and are principal
proper ideals of  the valuation domain $D_{Q}$ and hence are in $QD_{Q}$.
 By Proposition \ref{D} above, 
 $QD_{Q}=Q$, so, for each $y$ in $Q$, $yD_{Q}$ is (also) an ideal of $D$, i.e., $yD_{Q}=yD$.
  Now, let 
 $x_{1},x_{2},\dots, x_{n}\in Q$ and consider  the ideal $(x_{1},x_{2},\dots, x_{n})$. 
 Since $
D_{Q}$ is a valuation domain, $(x_{1},x_{2}, \dots, x_{n})D_{Q}=dD_{Q}$ and we can
assume that $d$ is in $D$.
 So, for some $r_{i}\in D$ and $s_{i}\in
D\backslash Q$ we have $x_{i}=\frac{r_{i}}{s_{i}}d$, for each $i$.

 So $(x_{1},x_{2},\dots, x_{n})=
 (\frac{r_{1}}{s_{1}}d,\frac{r_{2}}{s_{2}}d,\dots,\frac{r_{n}}{s_{n}}d)$. 
Removing the
denominators, we get $s(x_{1},x_{2},\dots, x_{n})=(t_{1}d, t_{2}d,\dots,t_{n}d)=(t_{1},t_{2},\dots,t_{n})d$, for some $s \in D\setminus Q$ , where $s_i|s$ and $t_i := \frac{s}{s_i}r_i$ , for each $i$.
As $dD_{Q}=(x_{1},x_{2},\dots, x_{n})D_{Q}=s(x_{1},x_{2},\dots, x_{n})D_{Q} $ $=(t_{1},t_{2},\dots, t_{n})dD_{Q}
$, we conclude that $(t_{1},t_{2},\dots, t_{n})D_{Q}=D_{Q}$. 
But that means that
at least one of the $t_{i}$ is in $D\backslash Q$ 
and hence is a comparable
element (Lemma \ref{C}(5)). 
But then, by \cite[Theorem 2.4]{[GMZ]}, $(t_{1},t_{2},\dots, t_{n})$ is
principal generated by a comparable element $t$.
 Thus, 
 $s(x_{1}, x_{2},\dots, x_{n})=(t_{1},t_{2},\dots, t_{n})d =tdD$. Since $s$ and $t$ are comparable, we have two
possibilities: {\bf ($\boldsymbol{\alpha} $)}\ $u(x_{1},x_{2},\dots, x_{n})=dD$ or {\bf ($\boldsymbol{\beta}$)}\ $
(x_{1},x_{2},\dots, x_{n})=vdD$, for some $u, v \in D$. In both cases $(x_{1},x_{2},...x_{n})$ turns
out to be a principal ideal of $D$ (in case {\bf ($\boldsymbol{\alpha}$)} because $d\in u(x_{1},x_{2},\dots, x_{n})$ and so  $u|d$ in $D$).
\end{proof}
%%%%%%%%%%%

\medskip 

%%%%%%%%%%%%%%%%%%%%%%%%%%%%%
%%% SECTION 6
%%%%%%%%%%%%%%%%%%%%%%%%%%%%
%%%%%%%%%%%%%%%%%%%%%%%%%%%%
\section{Applications: Shannon's quadratic extension}
%%%%%%%%%%%%%%%%%%%%%%%%%%%%
%%%%%%%%%%%%%%%%%%%%%%%%%%%%

 A domain $D$ is a {\it treed domain} if it has a treed spectrum, i.e., 
Spec($D$) is  a tree  as a poset   with respect to the set inclusion.  Note that $D$ is
a treed domain if and only if any two incomparable primes of $D$ are co-maximal.
Indeed, if $D$ is a treed then  $D_{P}$ is also a treed (more precisely, Spec$(D_P)$ is linearly ordered)
for every nonzero prime ideal $P$ of $D$. So, by Proposition \ref{AD},  $D_P$ is a $t$-local domain and thus $P = PD_P \cap D$ is a $t$-ideal of $D$.
 Indeed, if $F$ is a finitely generated ideal of $D$ contained in $P$, then $F^tD_P =F^vD_P\subseteq (FD_P)^v  = (FD_P)^t \subseteq (PD_P)^t = PD_P$ 
and so $ F^t \subseteq (FD_P)^t \cap D \subseteq PD_P \cap D = P$ (see  also \cite[page 436]{[Z-t]}).
Therefore, in a treed domain, every
nonzero prime ideal  is a $t$-ideal (Proposition \ref{AD}), in particular every maximal ideal is a $t$-ideal, and moreover it is well behaved.
However, a general $t$-local domain $D$ may
not have Spec($D$)  a tree  as, for instance, Examples \ref{bad behaved 2}  and \ref{DB} indicate.   So the class of treed domains   is strictly
contained in the class of domains whose maximal ideals are $t$-ideals. But,
in the presence of some extra conditions, this distinction may disappear.

%%%%%%%%%%%%%% PROPOSITION \label{H}
\begin{proposition} \label{H}
For a Pr\"ufer $v$-multiplication domain $D$, the following conditions are
equivalent.
\begin{enumerate}
\item[(i)] 
Every maximal ideal of $D$ is a $t$-ideal.

 \item[(ii)] 
Every prime ideal of $D$ is a $t$-ideal.

\item[(iii)] {\em{Spec}($D$)} is a tree.

\item[(iv)] $D$ is a Pr\"ufer domain.
\end{enumerate}
\end{proposition}
%%%%%%%%%%
\begin{proof}
(iv) $\Rightarrow$ (iii) $\Rightarrow$ (ii) $\Rightarrow$ (i) hold in general (without the P$v$MD assumption). More precisely, 
 (iv) $\Rightarrow$ (iii) is clear because in a Pr\"ufer domain 
$D$,  $D_{P}$ is a valuation domain for every nonzero prime ideal $P$ and so Spec($D$) is a tree.   (iii) $\Rightarrow$ (ii) has been explained above.

 (i) $\Rightarrow$ (iv)  For every prime $t$-ideal $P$ of a P$v$MD $D$,
we have $D_{P}$ a valuation domain (see, for instance, \cite[Corollary 4.3]{[MZ]}) and  if we assume that $D_{M}$ is a
valuation domain, for every maximal ideal $M$ of $D$, then $D$ is well known to be
a Pr\"ufer domain. 
\end{proof}
%%%%%%%%%%

The previous proposition  leads to
the following result for FC domains.

%%%%%%%%%% COROLLARY \label{HA}
\begin{corollary} \label{HA} 
Let $D$ be an integral domain. The following are equivalent.
\begin{enumerate}
\item[(i)] $D$is  an integrally closed finite conductor treed domain. 
\item[(ii)] $D$ is a treed P$v$MD;
\item[(iii)] $D$ is Pr\"ufer.
\end{enumerate}
\end{corollary}
%%%%%%%%
\begin{proof} (i) $\Rightarrow$ (ii), since an integrally closed finite conductor domain is a P$v$MD by Proposition \ref{FA} and  \cite[Corollary 4.3]{[MZ]}.
(ii) $\Leftrightarrow$ (iii) by Proposition \ref{H} and (iii) $\Rightarrow$ (i) because a Pr\"ufer domain is a FC domain \cite[Corollary 10]{[Z-FC]}.
\end{proof}
%%%%%%%%%%%%%%%%%%%%%%%%% 

\medskip

Indeed, it is worth noting that a nonzero proper ideal $I$ in an integral domain $D$
is said to be an {\it ideal of grade $1$} if  $I$ does not contain a set of
elements forming a regular sequence of length $\geq 2$.
Recall that, if an ideal $I$ of an integral domain $D$ contains a regular sequence of length 2, then $I^{-1}
 = D$   \cite[Exercise 1, page 102]{[Kap]}. 
 So, every $t$-ideal of an integral domain
is a grade 1 ideal and every nonzero prime ideal in a treed domain is a
grade 1 ideal.
With this background, for the next application we need  a little bit of preparation.
\smallskip 

Let $(R, \frak{m})$ be a regular
local integral domain with quotient field $F$  and $\frak{p}$ a prime ideal of $R$ so that $R/\frak{p}$ is a regular local domain.
A {\it monoidal transform of $R$ with nonsingular center} $\frak{p}$ 
is a local domain of the type $T:=R[\frak{p}x^{-1}]_Q$,   where $0\neq x \in \frak{p}$ and $Q $ is a prime ideal in $R[\frak{p}x^{-1}]$ such that $\frak{m} \subseteq Q$.
In particular, assume that $ \dim(R)=n$,  and $
\frak{p}=\frak{m}=(x_{1},x_{2},\dots,x_{n})R$, where $\{x_{1},x_{2},\dots, x_{n}\}$ form a regular sequence in $R$. 
Choose $i\in \{1, 2,\dots,n\}$, and consider the overring $R[x_{1}/x_{i},x_{2}/x_{i},\dots ,x_{n}/x_{i}]$  of $R$. Take
any prime ideal $Q$ of $R[x_{1}/x_{i}, x_{2}/x_{i},\dots,$ $x_{n}/x_{i}]$ such that 
$Q\supseteq \frak{m}$. 
The ring $R_{1}:=R[x_{1}/x_{i},x_{2}/x_{i},\dots ,$ $x_{n}/x_{i}]_{Q}$ is called a {\it
local quadratic transform} (for short,  {\it  LQT}) of $R$, and, again,  $R_1$ is a regular local integral domain with maximal ideal $\frak{m}_1:=
QR[x_{1}/x_{i},x_{2}/x_{i},\dots ,$ $x_{n}/x_{i}]_{Q}$   \cite[Corollary 38.2]{[Na]}.
Assume that $\dim(R)\geq 2$ in order to have that $R\neq R_1$. 
By Cohen's dimension inequality formula  $\dim(R_{1})\leq n$ \cite[Theorem 15.5]{[Matsu]} (and, more precisely, $\dim(R_{1})=n$ if and only if  $R_{1}/\frak{m}_1$ is an algebraic extension of $R/\frak{m}$)    \cite[(1.4)]{[Ab-1966]}. 

If we iterate the process, we obtain a sequence $
R=:R_{0}\subseteq R_{1}\subseteq R_{2}\subseteq ...$ of regular local
overrings of $R$ such that for each $j \geq 0$, $R_{j+1}$ is a LQT of $R_{j}$.
After a finite number of iterations, the sequence of nonincreasing integers $\dim(R_{j})$ is necessarily bound to stabilize, and
this process of iterating LQTs of the same Krull dimension (definitively, after a certain point)  and ascending
unions of the resulting regular sequences are of interest in algebraic geometry.
 For a description the reader may consult   a couple of recent papers    \cite{[Heinzer et al]} and  \cite{[HLOST]}.
So, let $R=:R_{0}\subseteq R_{1}\subseteq R_{2}\subseteq \dots$ be a sequence of
LQTs from a regular local integral domain $R$ with $\dim(R) \geq 2$ and $\dim(R_j) \geq 2$, for each $j\geq 1$, as described above. The ring $S:=\bigcup _{j\geq 0}R_{j},$ dubbed in
recent work as  {\it Shannon's Quadratic Extension of} $R$, to honor David Shannon \cite{[Sh]}
for his interesting contribution, has drawn particular attention.

Briefly, before Shannon, Abhyankar \cite[Lemma 12]{[Ab]} had shown that, if the regular local ring $R$ has dimension 2, then $
S$ is a valuation overring of $R$ such that the maximal ideal $\frak{m}_{S}$ of $S$
contains the maximal ideal $\frak{m}$ of $R$. David Shannon,  one of Abhyankar's students,  showed that if $\dim(R) >2$, 
$S$ need not be a valuation ring \cite[Examples 4.7 and
4.17]{[Sh]}.

Our purpose here is to look at $S$ from a simple star-operation theoretic
perspective,  to provide some direct straightforward and brief proofs of some
known results and point to known results that could simplify some of the
considerations in recent work.

We start by gathering some information about  the Shannon's Quadratic Extension $S$. 
Next two properties can be easily proved.

\begin{enumerate}

\item[(1)] {\it $S(:=\bigcup _{j\geq 0}R_{j})$, as described above, is a  local ring and, if
$\frak{m}_{S}$ denotes the maximal ideal of $S$,
  $\frak{m}_{S}=\bigcup _{j\geq 0}\frak{m}_{j} $ where $\frak{m}_{j}$ is the maximal ideal of the LQT\ $R_{j}$.}

\item[(2)]  {\it $S$ is integrally closed, as being integrally closed  a first order
property which is preserved by directed unions and hence, in particular, by ascending unions.}

\end{enumerate}

Since $S$ is directed union of regular local integral domains and,   by the Auslander-Buchsbaum theorem \cite[Theorem 20.3]{[Matsu]}, each regular local integral domain is a UFD  and hence, in particular,  a GCD
domain and so, {\sl a fortiori,} a Schreier domain. This observation gives us the next property of $S$.

\begin{enumerate}

\item[(3)] {\it
 $S$ is (at least) a Schreier domain.}  
 \end {enumerate}

 This follows from  a direct verification   that a direct union of (pre-)\-Schreier domains is a  (pre-)Schreier domain.

 %%%%%%%%%%%% REMARK
 \begin{remark}
 \em{
 Note that it is not true that a direct union of GCD-domains is a GCD-domain. An example   can be given by an integral domain of the type  $D^{(\Sigma)}:= D+XD_\Sigma[X] = \bigcup \{D[X/s] \mid s \in \Sigma \}$, where $D$ is a GCD domain and $\Sigma$ is a saturated multiplicative subset $D$, 
 since   it is known that} $D^{(\Sigma)}$ is not a GCD if   $\Sigma$ is not a splitting set,  i.e.,  if $\Sigma$ does not verify the condition that, for each $0\neq d\in D$, $d=sa$ for some $s\in \Sigma$ and $a\in D$ with $aD\cap s'D=as'D$ for all $s'\in \Sigma$  \cite [Corollary 1.5]{[Z-D_S]}.

 We give now an explicit  example.
 Let $\boldsymbol{\mathcal{E}}$ be the ring of entire functions. It is well known that $\boldsymbol{\mathcal{E}}$
is a B\'ezout domain    \cite[Exercise 18, page 147]{[G]}
and that every nonzero nonunit $x$ of $\boldsymbol{\mathcal{E}}$ can be written
uniquely as a countable product of finite powers of non associate primes,
i.e., $x=u\prod_{\alpha \in A}p_{\alpha }^{n_{\alpha }}$ where $A$ is a countable set, $n_{\alpha
}$ are natural numbers and $p_{\alpha }$ are mutually non associated primes elements of $\boldsymbol{\mathcal{E}}$
and $u$ is a unit in $\boldsymbol{\mathcal{E}}$.   The last property follows from the fact that the set of zeros of a nontrivial entire function is discrete, including multiplicities, the multiplicity of a zero of an entire function is a positive integer and a zero of an entire function determines a principal prime in $\boldsymbol{\mathcal{E}}$  \cite[Theorem 6]{[Hel]}. Clearly,  each of these primes generate a height one maximal ideal of $\boldsymbol{\mathcal{E}}$   \cite[Exercise 19, page 147]{[G]}.

Let $\Sigma$ be the multiplicative set generated by all of these principal,
height one primes and let $X$ be an indeterminate. Then, the ring 
$
\boldsymbol{\mathcal{E}}^{(\Sigma)}:=\boldsymbol{\mathcal{E}}+X\boldsymbol{\mathcal{E}}_{\Sigma}[X]  = \bigcup \{\boldsymbol{\mathcal{E}}[X/s] \mid s \in \Sigma\}$ is not a GCD domain, even though $\boldsymbol{\mathcal{E}}[X/s] $ is a GCD domain for each $s \in \Sigma$.

Indeed, if $x \in \boldsymbol{\mathcal{E}}$ is an infinite product of primes then it is not possible
to write $x =sx_{1}$ where $s\in \Sigma$ and $x _{1}$ is not
divisible by any of the  nonunits in $\Sigma$, since each $s$ is a finite product
of primes and $x $ is a product of infinitely many primes from $\Sigma$.
Thus, $\Sigma$ is not a splitting set and so $\boldsymbol{\mathcal{E}}^{(\Sigma)}$ cannot be a GCD domain. 

However, we claim that
 $\boldsymbol{\mathcal{E}}^{(\Sigma)}$ is a locally
GCD domain.
For  proving the claim, we need some preliminaries.
A prime ideal $P$ of an integral domain $D$ is said {\it to intersect in detail a multiplicative set} $\Sigma$ of $D$
if every nonzero prime ideal $Q$ contained in $P$ intersects $\Sigma$.
It was shown \cite[Proposition 4.1]{[Z-D_S]} that if $D$ is a locally GCD domain
and $\Sigma$ is a multiplicative set of $D$ such that every maximal ideal of $D$
that intersects $\Sigma$, intersects $\Sigma$ in detail, then $D^{(\Sigma)}$ is a locally
GCD domain. 

Indeed, clearly the B\'ezout domain $\boldsymbol{\mathcal{E}}$ is a locally  GCD domain. Moreover, as every maximal ideal
of $\boldsymbol{\mathcal{E}}$ that intersects $\Sigma$ contains a finite product of principal primes and
so must be a principal ideal. Thus,  every maximal ideal of $\boldsymbol{\mathcal{E}}$ that
intersects $\Sigma$, intersects it in detail. Consequently $\boldsymbol{\mathcal{E}}^{(\Sigma)}$ is a locally
GCD domain; however, $\boldsymbol{\mathcal{E}}^{(\Sigma)}$ is not a P$v$MD, since  $\boldsymbol{\mathcal{E}}^{(\Sigma)}$ is a Schreier domain and a P$v$MD which also is a Schreier domain is a GCD domain   \cite[Proposition 2.3]{[AZ-2007]}.

As a final remark, we  recall from   \cite[Proposition 4.3]{[Z-D_S]}  that   in a locally GCD  non-P$v$MD $D$ there always exists a maximal $t$-ideal $Q$ of $D$ such that $QD_Q$ is not a $t$-deal of $D_Q$. More precisely, it can be shown that an integral domain $D$ is a P$v$MD if and only if $D$ is locally P$v$MD and, for every $t$-prime ideal $P$ of $D$, $PD_P$  is a (maximal) $t$-ideal of $D_P$ 
 \cite[Corollary 4.4]{[Z-D_S]}.

 \end{remark}
 %%%%%%%%%%%%%%%%%%%%%%

We now resume our study of Shannon's Quadratic Extension $S$.
 \begin{enumerate}
\item[(4)] {\it There exists an element $x\in \frak{m}_{S}$
such that $\frak{m}_{S}=\sqrt{xS}$ }   \cite[Proposition 3.8]{[HLOST]}. 
 \end{enumerate}

\smallskip

The last property gives us, in light of Corollary \ref{A}(1), the
following property  that is of interest to us.

\begin{enumerate}
\item[(5)] {\it $S$ is a $t$-local integral domain}.

\end{enumerate}

This is enough information to provide very naturally the statements and      easy new
proof(s) of \cite[Theorem 6.2] {[Heinzer et al]}.

%%%%%%%%%%%%% THEOREM  \label{K} 
\begin{theorem}\label{K} 
{\em  (L. Guerrieri, W. Heinzer, B. Olberding and M.
Toeniskoetter \cite[Theorem 6.2] {[Heinzer et al]})}
 Let $S$ be a quadratic Shannon
extension of a regular local integral domain $R$. Then, the following are equivalent.
\begin{enumerate}

\item[(i)] 
 $S$ is a valuation domain

\item[(ii)]  $S$ is coherent.

\item[(iii)]  $S$ is a finite conductor domain.

\item[(iv)]  $S$ is a GCD domain.

\item[(v)]  $S$ is a P$v$MD.

\item[(vi)]  $S$ is a $v$-finite conductor domain.
\end{enumerate}
\end{theorem}
%%%%%%%
\begin{proof}
The equivalence of (i) $\Leftrightarrow $ (ii) $\Leftrightarrow $ (iii)
comes from Corollary \ref{FA}. Now (i) $\Leftrightarrow $ (iv) $\Leftrightarrow $
(v) follow from Proposition \ref{F} and, as $S$ is Schreier (by (3)), (i) $\Leftrightarrow $
(vi) by Corollary \ref{FD}.
\end{proof}
%%%%%%%%%%

From Lemma \ref{FC}, Corollary \ref{FD}, and Theorem \ref{K} we easily deduce the following.

%%%% COROLLARY \label{KA} 
\begin{corollary} \label{KA}
Let $S$ be a quadratic Shannon
extension of a regular local integral domain $R$. If $S$ is not a valuation domain, then $S$ contains a pair of
elements $a, b$ such that $aS\cap bS$ is not a $v$-ideal of finite type.
\end{corollary}
%%%%%%%%%%%
\begin{proof}
 If, for each pair of   elements $a,b\in S$, we had that $aS\cap bS$ is a $v$-ideal of finite type, then  $S$ would be a GCD domain by Corollary \ref{GCD}, since  $S$ is a Schreier domain  (by point (3) above).  Therefore, $S$ would be a valuation domain by Theorem \ref{K}, which is not the case.
\end{proof}
%%%%%%%%%
 This corollary is significant with reference to the proof of the previous theorem
(Theorem \ref{K}) in that there are P$v$MDs $D$, such as Krull domains, that
contain elements $a, b$ such that $aS\cap bS$ a $v$-ideal of finite type, which may not be finitely generated.

\medskip

From \cite[Proposition 4.1]{[HLOST]},  we conclude that $S$ has another property of
interest.

\medskip

\begin{enumerate}

\item[(5)] {\it For each element $x\in \frak{m}_{S}$ such that $\frak{m}_{S}=\sqrt{xS}$, the integral domain $T:=S[1/x]$ is
a regular local ring with $\dim(T)=\dim(S) -1$.
}
\end{enumerate}

\medskip

\ So, if $\dim(S)=2$ and $\frak{m}_{S}$ contains a nonzero comparable element then we know that $S
$ is a valuation domain (Theorem  \ref{D} and (5)).

\smallskip

If $\dim(S) > 2$ then $S$ cannot be a valuation
domain, whether $S$ contains a comparable element or not, because a regular local
ring $T$, constructed from $S$ as in (5),  has $\dim(T) > 1$, and thus $T$ may not be a valuation domain. 
However,  if $\frak{m}_{S}=pS$
is principal then, $S$ is a non-valuation $t$-local domain that contains a
comparable element, by Proposition \ref{prime}(2). 
This fact, together with   Proposition \ref{G}, provides a
definitive criterion that can be used to construct examples of non-valuation $t
$-local domains containing a comparable element, even in dimension two.

%%%%%%%%  EXAMPLE
\begin{example}
\em{
Let  $\mathbb Z$ be   the ring of integers, $\mathbb Q$ (resp., $\mathbb R$) the field of rational numbers (resp. real numbers) and $p$ a 
prime element in $\mathbb Z$.  Let $P$ be the maximal ideal of the DVR \ $\mathbb{R}[\![X]\!]$ and set $D:={\mathbb Z}_{(p)}+X\mathbb{R}[\![X]\!] = {\mathbb Z}_{(p)}+P$.
The integral domain $D$ is local with principal maximal ideal $M:= pD$ and $\bigcap _{n\geq 0} p^nD = X\mathbb{R}[\![X]\!]=P$.   Clearly, $p$ is a proper comparable element in $D$.
Since  $D_{P}=\mathbb{Q}+X\mathbb{R}[\![X]\!]$  is not a valuation
domain, $D$ is a 2-dimensional non-Noetherian non-valuation
$t$-local integral domain with prime spectrum linearly ordered given by $\{ M \supset P \supset (0)\}$.
}
\end{example}
%%%%%%%%%%%%%%%%%%%%%%%%%%%%%%%

In the same vein, and this is suggested by Tiberiu Dumitrescu, we
have another example.

%%%%%%%%%% EXAMPLE
\begin{example}
\em{
Let  $\mathbb Z$ be   the ring of integers, $\mathbb Q$  the field of rational numbers  and $p$ a nonzero
prime element in $\mathbb Z$.  Let $D:={\mathbb Z}_{(p)}+P$ where $P$ is the maximal ideal $(X^{2},X^{3})$ of
 $\mathbb{Q}[\![X^{2},X^{3}]\!]$.  As above, $D$ is  a local domain with maximal ideal $M = p{\mathbb Z}_{(p)}+P= pD$ and
   $\bigcap _{n\geq 0} p^nD =P$. 
   In  this case,  $D_{P}=\mathbb{Q}[\![X^{2},X^{3}]\!]$ which is a well known  
 1-dimensional Noetherian domain that is not a valuation domain (in fact, it is non integrally closed).
 Thus, $D$ is a 2-dimensional non-Noetherian non-valuation
$t$-local integral domain,   having  a proper comparable element and prime spectrum linearly ordered given by $\{ M:=p{\mathbb Z}_{(p)}+(X^{2},X^{3})\mathbb{Q}[\![X^{2},X^{3}]\!]\supset P \supset (0)\}$.
  
  We can provide examples in any dimension.
  Let $P$ be the maximal ideal of the $n$-dimensional regular local
ring $\mathbb{Q}[\![X_{1},X_{2},\dots, X_{n}]\!].$ Then $D:=\mathbb{Z}_{(p)}+P$ is local with maximal ideal $M:= pD$. In particular, $D$ contains a proper
comparable element, e.g., $p$, and, of course, $D_{P}$ is far from being a valuation
domain. Thus, $D$ is an $(n+1)$-dimensional non-valuation
$t$-local integral domain.
}
\end{example}
%%%%%%%%%%%%%%%%

Note that a 1-dimensional domain that contains a
nonzero nonunit comparable element is a valuation domain. 
This follows from
the following two facts
(1) the presence of a comparable element forces the domain
to be (1-dimensional) $t$-local and
(2) a domain is a valuation domain if
and only if every nonzero prime ideal contains a nonzero comparable element (Lemma \ref{Kap}).

\medskip

 From (5), we deduce another interesting property of $S$.
\begin{enumerate}

\item[(6)] {\it Let $S$ be as above (i.e.,   a quadratic Shannon
extension of a regular local integral domain), for each element $x\in \frak{m}_{S}$ such that $\frak{m}_{S}=\sqrt{xS}$,
call the saturation of the multiplicative set $\{x^{n} \mid n\in \mathbb N\}$, {\em{span of}} $x$ and denote it
by \texttt{span}$(x)$.
 Then,}
 \begin{enumerate}
 \item[(6a)] {\it 
  for every nonunit $
h$ in \texttt{span}$(x)$ we have $\frak{m}_{S}=\sqrt{hS}$} \  and 

 \item[(6b)] {\it 
 $\frak{m}_{S}$ is generated by
nonunits in \texttt{span}$(x)$.}

\end{enumerate}
\end{enumerate}

 The saturated multiplicative set \texttt{span}$(x)$ has been used before, by Dumitrescu, Lequain, Mott, and Zafrullah in \cite{[DLMZ]}, to determine the number of distinct maximal $t$-ideals that the element  $x$ belongs to. 
Here, the statement that the ideal $\frak{m}_{S}$  is generated by nonunit members of \texttt{span}$(x)$   is caused by the fact that there is only one maximal $t$-ideal (i.e., $\frak{m}_{S}$) involved.

\medskip

  Note that, before introducing quadratic Shannon extensions of local regular rings, all examples of $t$-local domains that we have considered in the present paper were valuation domains or rings obtained by some pullback construction.  At this point, it is natural to ask if    the quadratic Shannon extensions, that are not valuation domains, could as well be obtained  by some appropriate pullback construction.
For this purpose, we start by recalling some other properties of the quadratic Shannon extensions.
\begin{enumerate}

\item[(7)] {\it Let $S$ be as above (i.e.,   a quadratic Shannon
extension of a regular local integral domain of dimension $> 2$). If $S$ is Archimedean, then its complete integral closure $S^\ast$ coincides with $(\frak{m}_S : \frak{m}_S) =T \cap W$, where $\frak{m}_S$ is the maximal ideal of $S$, $T =S[1/x]$ is the local regular overring of $S$ introduced in {\em{(5)}} and $W$ is a uniquely determined valuation overring of $S$ and if $S \neq S^\ast$, $S^\ast$ is a generalized Krull domain \cite[Theorem 6.2]{[HLOST]}.
 }
\end{enumerate}

In the previous situation, if $S \neq S^\ast$,  $\frak{m}_S$ is a height 1 prime ideal of $S^\ast$, since it is the center of the maximal ideal of the valuation overring $W$ of $S^\ast$  
(see \cite[Corollary 6.3]{[HLOST]} and \cite[Theorem 7.4]{[HOT]}). Therefore, $S$ is the pullback of the residue field $S/\frak{m}_S$ with respect to the canonical projection $S^\ast \rightarrow S^\ast /\frak{m}_S$.

On the other hand, in the non-Archimedean case, we know the following fact.
\begin{enumerate}

\item[(8)] {\it Let $S$ be as above (i.e.,   a quadratic Shannon
extension of a regular local integral domain of dimension $> 2$). If $S$ is non-Archimedean, then its complete integral closure $S^\ast$ coincides with the overring $T =S[1/x]$, $\bigcap\{ x^nS \mid n\geq 0 \}=:\frak{p}$ is a proper prime ideal of $S$ and $T=(\frak{p}:\frak{p})$ \cite[Threorem 6.9 and Corollary 6.10]{[HLOST]}.
}  
\end{enumerate}

In the previous situation,  the integral domain $S/\frak{p}$ is a DVR \cite[Lemma 3.4]{[HLOST]},  and $T =S_\frak{p}$, since $T=S[1/x]$ is a ring of fractions of $S$ and $\frak{p}$ is disjoint from the multiplicative set $\{ x^n \mid n \geq 0\}$. Therefore, $S$ is the pullback of  $S/\frak{p}$ with respect to the canonical projection $T \rightarrow T/\frak{p}$, where $T/\frak{p}$ is a field, coinciding with the residue field $S_\frak{p}/\frak{p}S_\frak{p}$   (isomorphic to the field of quotients of the integral domain $S/\frak{p}$).

\medskip

The last remaining case is when the quadratic Shannon
extension $S$ is (Archime\-dean and) completely integrally closed. An example is given in  \cite[Corollary 7.7]{[HOT]}. 
In this situation $S$ may not be obtained by a pullback construction of some of its overrings, since, if an integral domain $A$ shares a nonzero ideal with one of its  proper  overrings $B$ then $A$ and $B$ must have the same complete integral closure  \cite[Lemma 5]{[GH]}.

\medskip

We end with a classification of the $t$-local domains, which could be useful for  detecting  $t$-local domains that are not issued from a pullback construction.
The following proposition is a consequence of  more general results concerning $DT$-domains, proved by  G. Picozza and  F. Tartarone in  \cite{[PT]}.

%%%%%%%%%%%%%%%%%  PROPOSITION
\begin{proposition} Let $(D,M)$ be a local domain.
\begin{enumerate}
\item[(1)]  If  $D\neq (M:M)$, then $D$ is a $t$-local domain.

\item[(2)]  If  $D= (M:M)$ and $M$ is finitely generated, then $D$ is a $t$-local domain if and only if $M$ is principal.

\item[(3)]  If  $D=  (M:M)$, and $M$ is not  finitely generated, then $D$ is a $t$-local domain if and only if $M$ is not $t$-invertible.
\end{enumerate} 
\end{proposition}
%%%%%%%%%%%%%%%%%
\begin{proof} (1) If $D\neq (M:M)$, then necessarily the maximal ideal $M$ is the conductor of the inclusion $D \hookrightarrow (M:M)$ and so $M$ is a divisorial ideal of $D$.

(2) Assume that $D=  (M:M)$, and $M$ is  finitely generated, clearly $M$ is divisorial if and only if $(M:M)=D \neq M^{-1}= (D:M)$ and this happens if and only if $M \neq MM^{-1} (\subseteq D)$ or, equivalently, if and only if $MM^{-1}=D$. In a local domain, a nonzero ideal is invertible if and only if it is a principal ideal.

(3) Assume that $D=  (M:M)$, $M$ is not  finitely generated and, moreover, $M$ is not a $t$-invertible ideal. If $M$ is not a $t$-ideal, then $M^t=D$ and thus $(MM^{-1})^t= M^t =D$, which is a contradiction.

Conversely, since $M$ is  not finitely generated, $M$ is not invertible and, since  $D$ is $t$-local,  $M$ is  not  even $t$-invertible  (Theorem \ref{t-inv} ((iii)$\Rightarrow$(vi)).
\end{proof}
%%%%%%%%%%%%%%%%%%%

Any pseudo-valuation non-valuation domain provides an example of case (1); a   discrete valuation domain  (for short, DVR) is an example of case (2) and a rank 1  non-DVR valuation domain is an example of case (3).

\medskip

%%%%%%%%%

\noindent {\bf Acnowledgments.} The authors would like to thank Francesca Tartarone and  Lorenzo Guerrieri for the useful conversations on some aspects of the present paper and the anonymous
referee for several helpful suggestions which  improved the quality of the manuscript.
%%%%%%%

%%%%%%%%%%%   {thebibliography}{AB}

\end{document}